\documentclass{amsart}
\usepackage{amsmath, amsthm, amssymb}
\usepackage{amsmath, amssymb}
 \usepackage{amsmath,amscd,amsthm}
\usepackage{mathrsfs}
\usepackage{tikz}
\input xypic
\xyoption{all}

\theoremstyle{plain} 
\newtheorem{theorem}[subsection]{Theorem}
\newtheorem{proposition}[subsection]{Proposition}
\newtheorem{lemma}[subsection]{Lemma}
\newtheorem{corollary}[subsection]{Corollary}

\theoremstyle{definition}

\theoremstyle{remark}
\newtheorem{remark}[subsection]{Remark}
\newtheorem*{remark*}{Remark}

\numberwithin{equation}{subsection}

\DeclareMathOperator{\cone}{{\mathrm{cone}}}

\begin{document}

\title[Local Finiteness of the Twisted Bruhat Orders]{Local Finiteness of the Twisted Bruhat Orders\\ on Affine Weyl Groups}

\author{Weijia Wang}
\address{School of Mathematics (Zhuhai)
\\ Sun Yat-sen University \\
Zhuhai, Guangdong, 519082 \\ China}
\email{wangweij5@mail.sysu.edu.cn}
\date{\today}

\begin{abstract}
In this paper we investigate various properties of  strong and weak twisted Bruhat orders on a Coxeter group. In particular we prove that any  twisted strong Bruhat order on an affine Weyl group is locally finite, strengthening a result of Dyer in J. Algebra, 163, 861--879 (1994). We also show that for a non-finite and non-cofinite biclosed set $B$  in the positive system of an affine root system with rank greater than 2, the set of elements having a fixed $B$-twisted length is infinite. This implies that the  twisted strong and weak Bruhat orders have an infinite antichain in those cases. Finally we show that  twisted weak Bruhat order can be applied to the study of the tope poset of an infinite oriented matroid arising from an affine root system.
\end{abstract}

\maketitle

\section{Introduction}

Twisted strong and weak Bruhat orders on a Coxeter group were introduced by Dyer and the author respectively.
These generalizations of the ordinary strong and weak Bruhat orders have found connections with problems related to reflection orders,
representation theory and combinatorics of infinite reduced words.
In \cite{DyerTwistedBruhat} it is shown that Kazhdan-Lusztig type polynomials can be  defined for certain intervals of the  twisted strong
Bruhat orders and these polynomials are used to formulate a conjecture regarding the representation of Kac-Moody Lie algebras.
In \cite{Gobet},  twisted strong Bruhat order is studied for the twisted filtration of the Soergel bimodules.
In \cite{chenyu}, twisted strong Bruhat order is related to the poset of $B\times B$-orbits on the wonderful compactification of algebraic groups.
In \cite{orderpaper}, the semilattice property of twisted weak Bruhat order is shown to characterize the biclosed sets arising from the infinite reduced words in affine cases.

In this paper, we prove several properties regarding the structure of twisted  Bruhat orders which are not previously known. First we study the finite intervals of twisted strong Bruhat orders. Finite intervals are of particular importance as the Kazhdan-Lusztig type polynomials in \cite{DyerTwistedBruhat} can only be defined for finite intervals with other favorable properties. It is not previously known whether for any infinite biclosed set $B$, the $B$-twisted strong Bruhat order on an affine Weyl group is locally finite, i.e. any interval of this poset is finite. In \cite{quotient}, Dyer gives a partial answer to the question, providing a technical condition guaranteeing the local finiteness. In section \ref{localfinite} we solve the problem in the positive, showing that such a poset is always locally finite. Indeed we prove a stronger fact for affine Weyl groups: the set $\{y|y\leq_B x, l_B(x)-l_B(y)=n\}$ is finite for any $x$ and biclosed set $B$. The proof exploits the explicit description of biclosed sets in the affine root system first given in \cite{DyerReflOrder}.

In Section \ref{hyperinterval}, we consider the question whether for some biclosed set $B$, the strong $B$-twisted Bruhat order on a nonaffine Coxeter group is not locally finite. We propose a procedure to obtain an infinite interval for certain $B.$

In Section \ref{fixlength}, we show that while a twisted strong Bruhat order on an affine Weyl group is locally finite, the set of elements with a fixed twisted length is always infinite provided the twisting biclosed set is neither finite nor cofinite and the rank of the affine Weyl group is greater than 2. This result implies that the  twisted strong and weak Bruhat order has an infinite antichain in those cases. The proof makes use of the properties of twisted weak Bruhat order and an explicit description of the biclosed sets in the affine root system.

In Section \ref{examplesec}, we present the structure of a twisted strong Bruhat order on the affine Weyl group $\widetilde{W}$ of type $\widetilde{A}_2$. Such a  description of the structure of a  twisted Bruhat order (which is not isomorphic to ordinary Bruhat order or its opposite) was only known for $\widetilde{A}_1$ previously.

In Section \ref{omsec}, we study the tope poset of the infinite oriented matroid from an affine root system. We first describe all hemispaces (topes) of such an oriented matroid and then show that twisted weak Bruhat orders show up in the tope poset. These results allow us to prove that certain finite intervals in the tope poset are lattices.

\section{Preliminaries}

Let $B$ be a set. Denote by $|B|$ the cardinality of $B$. We denote the disjoint union of sets by $\uplus$.
We refer the reader to \cite{bjornerbrenti} and \cite{Hum} for the basic notions of Coxeter groups and their root systems. We call a Coxeter system without braid relations a universal Coxeter system.

\subsection{Biclosed Sets of Coxeter Groups}

Let $(W,S)$ be a Coxeter system. Denote by $T$ the set of reflections. Given a root $\alpha$, denote by $s_{\alpha}$ the corresponding reflection. Let $t\in T$ be a  reflection. Denote by $\alpha_t$ the corresponding positive root.
For $w\in W$, the inversion set of $w^{-1}$ is defined to be $\{\alpha\in \Phi^+|w^{-1}(\alpha)\in \Phi^-\}$ and is denoted by $N(w)$.
For a general Coxeter system $(W,S)$, $\Phi,\Phi^+,\Phi^-$ denote the set of roots, positive roots and negative roots respectively. We denote by $l(w)$ the (usual) length of $w\in W.$
A set $\Gamma\subset\Phi$ is 2 clousre closed if for any $\alpha,\beta\in \Gamma$ and $k_1\alpha+k_2\beta\in\Phi, k_1,k_2\in \mathbb{R}_{\geq 0}$ one has that  $k_1\alpha+k_2\beta\in\Gamma$. A set $B\subset \Gamma$ such that both $B$ and $\Gamma\backslash B$ are  closed  is called a  biclosed set in $\Gamma$. Finite biclosed sets in $\Phi^+$ are precisely inversion sets $N(x)$ for some $x\in W$ (\cite{DyerWeakOrder} Lemma 4.1(d)). If $s_1s_2\cdots s_k$ is a reduced expression of $x$, then $N(x)=\{\alpha_{s_1},s_1(\alpha_{s_2}),\cdots,s_1s_2\cdots,s_{k-1}(\alpha_{s_k})\}$.
Denote by $\widetilde{N}(x)=\{t\in T|\alpha_t\in N(x)\}$.
We say a positive root $\beta$ dominates another positive root $\alpha$ if $\beta\in N(w)$ for some $w\in W$ implies $\alpha\in N(w).$

Suppose that $W$ is infinite. An infinite sequence $s_1s_2s_3\cdots, s_i\in S$ is called an infinite reduced word of $W$ provided that $s_1s_2\cdots s_j$ is  reduced  for any $j\geq 1$. For an infinite reduced word $x=s_1s_2\cdots$, define the inversion set $N(x)=\cup_{i=1}^{\infty}N(s_1s_2\cdots s_i)$. Two infinite reduced words $x,y$ are considered equal if $N(x)=N(y)$. The inversion set of an infinite reduced word is   biclosed in $\Phi^+$.  A (finite) prefix of an infinite reduced word $w$ is an element $u$ in the Coxeter group $W$ such that $N(u)\subset N(w).$ The set of infinite reduced words of $(W,S)$ is denoted by $W_l$. For an inversion set $N(w)$ we denote by $N(w)'$ the complement of $N(w)$ in $\Phi^+$. An element $w\in W$ is said to be straight if $l(w^n)=|nl(w)|.$ It is shown that in \cite{speyer} that if $W$ is infinite any Coxeter element is straight. Choose a reduced expression $\underline{w}$ of a straight element $w$, then $\underline{w}\underline{w}\underline{w}\cdots$ well defines an infinite reduced word indepenent of the choice of $\underline{w}$ and we denote it by $w^{\infty}$.

There exists a $W$-action on the set of all biclosed sets in $\Phi^+$ given by $w\cdot B:=(N(w)\backslash w(-B))\cup (w(B)\backslash
(-N(w)))$. In particular $u\cdot N(v)=N(uv)$ for $u\in W, v\in W\cup W_l.$

We shall call a set $B$ in $\Phi^+$ cofinite if $\Phi^+\backslash B$ is finite.

\subsection{Twisted Bruhat Order and Twisted Weak Bruhat Order}

Let $B$ be a biclosed set in $\Phi^+$. The left (resp. right) ($B$-)twisted length of an element of $w\in W$ is defined as $l(w)-2|N(w^{-1})\cap B|$ (resp. $l(w)-2|N(w)\cap B|$) and is denoted by $l_B(w)$ (resp. $l'_B(w)$). Clearly $l_B(w)=l_B'(w^{-1})$ and $l_{\emptyset}(w)=l'_{\emptyset}(w)=l(w)$.

Fix a biclosed set $B$ in $\Phi^+$, the left (resp. right) ($B$-)twisted strong  Bruhat order $\leq_B$ on $W$ is defined as follow: $x\leq_B y$ if and only if $y=t_kt_{k-1}\cdots t_1x$ with $l_B(t_i\cdots t_1x)=l_B(t_{i-1}\cdots t_1x)+1, t_i\in T, 1\leq i\leq k$ (resp. $l_B'(xt_1\cdots t_{i})=l_B'(xt_1\cdots t_{i-1})+1, t_i\in T, 1\leq i\leq k$).
One easily sees that the left twisted strong  Bruhat order is isomorphic to the right twisted strong  Bruhat order under the map $w\mapsto w^{-1}$. The left (resp. right) $\emptyset-$twisted strong Bruhat order is the ordinary Bruhat order. For $x\in W$ it is known that the left (resp. right) $N(x)-$twisted strong Bruhat order is isomorphic to the ordinary Bruhat order. Following the convention in the literature, we consider the left twisted strong  Bruhat orders instead of the right ones and often refer to left twisted strong Bruhat order as twisted Bruhat orders. Left $B-$twisted strong Bruhat order can also be characterized equivalently as the unique partial order on $W$ with the following property:
for $t\in T, w\in W$, if $\alpha_t\in w\cdot B$, $tw\leq_B w$ and if $\alpha_t\not\in w\cdot B$ then $tw\geq_B w.$

A closed interval $[x,y]$ in a twisted (strong) Bruhat order is said to be spherical if any  length 2 subinterval of it contains 4 elements. It is known that for a closed spherical interval $[x,y]$, the order complex of the corresponding open interval $(x,y)$ is a sphere.

The left (resp. right) ($B$-)twisted weak Bruhat order $\leq_B'$ on $W$ is defined as follow: $x\leq_B' y$ if and only if $N(x^{-1})\backslash N(y^{-1})\subset B$ and $N(y^{-1})\backslash N(x^{-1})\subset \Phi^+\backslash B$ (resp. $N(x)\backslash N(y)\subset B$ and $N(y)\backslash N(x)\subset \Phi^+\backslash B$). The left $B$-twisted weak Bruhat order is isomorphic to the right $B$-twisted weak Bruhat order under the map $w\mapsto w^{-1}$. For $x\in W$ it is known that the left (resp. right) $N(x)-$twisted weak Bruhat order is isomorphic to the ordinary left (resp. right) weak order.
If $x\leq_B' y$ (under the left (resp. right) $B$-twisted weak Bruhat order), one has that $x\leq_B y$ (under the left (resp. right) $B$-twisted strong Bruhat order).
The left (resp. right) $\emptyset-$twisted weak Bruhat order is the ordinary left (resp. right) weak Bruhat order. Following the convention in the literature, we consider the right twisted weak  Bruhat orders instead of the left ones and often refer to right twisted weak Bruhat orders as twisted weak orders.

The (left) twisted (strong) Bruhat order (resp. (right) twisted weak (Bruhat) order) is graded by the left twisted length function (resp. right twisted length function). Suppose that $u\leq_B' v$ (under the right $B$-twisted weak  order). It is shown in \cite{orderpaper} that one has a chain: $$u=u_0<_B'u_1<_B'u_2<_B'\cdots <_B' u_t=v$$
such that $u_{i}s=u_{i+1}$ for some $s\in S$ and $l_B'(u_i)+1=l_B'(u_{i+1})$. We will refer to this property as the chain property of the twisted weak order and  write $u_i\vartriangleleft u_{i+1}$.

Let $B$ be an inversion set of an element in $W$. Then $(W,\leq_B')$ is isomorphic to $(W,\leq_{\emptyset}')$. It is known that $(W,\leq_{\emptyset}')$ is a complete meet semilattice (Chapter 3 of \cite{bjornerbrenti}). Let $B$ be the inversion set of an infinite reduced word.
It is shown in \cite{orderpaper} that $(W,\leq_B')$ is a non-complete meet semilattice
and for an affine Weyl group $W$, $(W,\leq_B')$ is a non-complete meet semilattice only if $B$ is the inversion set of an infinite reduced word.

\subsection{Construction of Affine Weyl Groups and Their Root Systems}

Let $W$ be an irreducible Weyl group with the crystallographic root system $\Phi$ contained in the Euclidean space $V$. For these notions, see Chapter 1 of \cite{Hum}. The root system of an (irreducible) affine Weyl group $\widetilde{W}$ (corresponding to $W$) can be constructed as the ``loop extension" of $\Phi$.
We describe such a construction. Let $\Phi^+$ be the chosen standard positive system of $\Phi$ and let $\Delta$ be the simple system of $\Phi^+.$

Define a $\mathbb{R}-$vector space $V'=V\oplus\mathbb{R}\delta$ where $\delta$ is an indeterminate. Extend
the inner product on $V$ to $V'$ by requiring $(\delta,v)=0$ for any $v\in V'$.
If
$\alpha\in \Phi^+$, define
$\{\alpha\}^{\wedge}=\{\alpha+n\delta|n\in \mathbb{Z}_{\geq 0}\}\subset
V'$. If $\alpha\in \Phi^-$, define
$\{\alpha\}^{\wedge}=\{\alpha+(n+1)\delta|n\in \mathbb{Z}_{\geq
0}\}\subset V'$.

For a set $\Lambda\subset \Phi$, define
$\Lambda^{\wedge}=\bigcup_{\alpha\in\Lambda}\{\alpha\}^{\wedge}\subset
V'$. The set of roots of the affine Weyl group $\widetilde{W}$, denoted by $\widetilde{\Phi}$, is $\Phi^{\wedge}\uplus-\Phi^{\wedge}$.
The set of positive roots  and the set of negative roots are $\Phi^{\wedge}$ and $-\Phi^{\wedge}$ respectively. The set of simple roots is $\{\alpha|\alpha\in \Delta\}\cup\{\delta-\rho\}$ where $\rho$ is the highest root in $\Phi^+$. Let $\alpha\in \widetilde{\Phi}$ be a root. The reflection in $\alpha$, denoted by $s_{\alpha}$, is an $\mathbb{R}-$linear map $V'\rightarrow V'$ defined by
$v\mapsto v-2\frac{(v,\alpha)}{(\alpha,\alpha)}\alpha.$
The (irreducible) affine Weyl group $\widetilde{W}$ is generated by
$s_{\alpha},\alpha\in \widetilde{\Phi}$. The simple reflections are the reflections in the simple roots of $\Phi^{\wedge}$.  For $v\in V$,  define the $\mathbb{R}-$linear map $t_v$ which acts on $V'$   by $t_v(u)=u+(u,v)\delta.$ For $\alpha\in \Phi,$ define the coroot $\alpha^{\vee}=2\alpha/(\alpha,\alpha)$. Note that $s_{\alpha+n\delta}=s_{\alpha}t_{-n\alpha^{\vee}}$.  Let $T$  be the free Abelian group generated by $\{t_{\gamma^{\vee}}|\gamma\in \Delta\}$. Then it is known that $\widetilde{W}=W\ltimes T.$ Let $\pi$ be the canonical projection from $\widetilde{W}$ to $W$.

If $\alpha\in \Phi^+$, let $(\alpha)_0=\alpha\in \Phi^{\wedge}$. If $\alpha\in \Phi^-$, let $(\alpha)_0=\alpha+\delta\in \Phi^{\wedge}$. We also introduce the following compact notation: $\alpha_a^b:=\{\alpha+k\delta|a\leq k\leq b\}$ for $a,b\in \mathbb{Z}\cup \{\pm\infty\}.$ We also denote $\cup_{i=1}^t(\alpha_i)_{a_i}^{b_i}$ by $(\alpha_1)_{a_1}^{b_1}(\alpha_2)_{a_2}^{b_2}\cdots(\alpha_t)_{a_t}^{b_t}$.

\subsection{Biclosed Sets of an Affine Weyl Group}\label{biclosedaffine}

Let $\Phi$ be a finite irreducible crystallographic root system with $\Phi^+$ the standard positive system and $\Delta$ the simple system. Suppose that $\Delta'\subset \Phi$. Denote by $\Phi_{\Delta'}$ the root subsystem generated by $\Delta'$. It is shown in \cite{biclosedphi} that the biclosed sets in $\Phi$ are those $(\Psi^+\backslash \Phi_{\Delta_1})\cup \Phi_{\Delta_2}$ where $\Psi^+$ is a positive system of $\Phi$ and $\Delta_1,\Delta_2$ are two orthogonal subsets (i.e. $(\alpha,\beta)=0$ for any $\alpha\in \Delta_1,\beta\in \Delta_2$) of the simple system of $\Psi^+.$ For simplicity  we denote the set $(\Psi^+\backslash \Phi_{\Delta_1})\cup \Phi_{\Delta_2}$ by $P(\Psi^+,\Delta_1,\Delta_2)$.

The biclosed sets in $\Phi^{\wedge}$ is determined in \cite{DyerReflOrder}.
Let $W'$ be the reflection subgroup of $\widetilde{W}$ generated by $(\Delta_1\cup\Delta_2)^{\wedge}$. Then any biclosed set in $\widetilde{\Phi}^+(=\Phi^{\wedge})$ is of the form $w\cdot P(\Psi^+,\Delta_1,\Delta_2)^{\wedge}$ for some $\Psi^+,\Delta_1,\Delta_2$ and $w\in \widetilde{W}.$
In particular a biclosed set $w\cdot P(\Psi^+,\Delta_1,\Delta_2)^{\wedge}$ where $w\in W'$ differs from $P(\Psi^+,\Delta_1,\Delta_2)^{\wedge}$ by finite many roots.

Suppose that $B$ is a biclosed set in $\widetilde{\Phi}^+=\Phi^{\wedge}$.
Let $I_B=\{\alpha\in \Phi|\,\{\alpha\}^{\wedge}\cap B$ is infinite$\}$ and
$A_B=\{\alpha\in \Phi|\,|\{\alpha\}^{\wedge}\cap B\neq \emptyset\}$. It is known that $I_B$ is biclosed in $\Phi$. Those biclosed sets $B$ such that $I_B=P(\Psi^+,\Delta_1,\Delta_2)$ are precisely the biclosed sets of the form $w\cdot P(\Psi^+,\Delta_1,\Delta_2)^{\wedge}$ where $w\in W'$. It is shown in \cite{wang} that if $B$ is an inversion set of an element of the affine Weyl group $\widetilde{W}$ or an infinite reduced word in $\widetilde{W}_l$, then $A_B\cap -A_B=\emptyset.$

The following theorem is proved in \cite{wang}

\begin{theorem}\label{infiniteword}
For an affine Weyl group, the biclosed sets which are inversion sets of infinite reduced words are precisely those of the form $w\cdot P(\Psi^+,\Delta_1,\emptyset)^{\wedge}, w\in \widetilde{W},\Delta_1\subsetneq \Delta$.
\end{theorem}

\section{Twisted Bruhat Orders on an Affine Weyl Group Are Locally Finite}\label{localfinite}

Throughout this paper unless otherwise specified, $W$ is a finite irreducible Weyl group and $\widetilde{W}$ is  the corresponding irreducible affine Weyl group. Denote by $\Phi$  the crystallographic root system of $W$ and denote by $\Phi^+$  a chosen standard positive system. Let $\alpha\in \Phi$. A finite $\delta$-chain through $\alpha$ is a set $\{\alpha+d\delta|p\leq d\leq q\}$ where $p\leq q$ are two integers.

\begin{lemma}\label{dominance}
(1) Suppose that $\alpha\in \Phi$ and $k\in \mathbb{Z}_{\geq 0}$. Then in $\widetilde{\Phi}^+$, $\alpha_0+k\delta$ dominates the roots $\alpha_0,\alpha_0+\delta,\cdots,\alpha_0+(k-1)\delta.$

(2) Let $w\in \widetilde{W}$. Then $w$ carries a finite $\delta$-chain through $\alpha\in \Phi$ to another finite $\delta$-chain through some $\beta\in \Phi.$ In particular $\beta=\pi(w)(\alpha)$
\end{lemma}

\begin{proof}
(1) is well known. See for example  Lemma 4.1 in \cite{wang}. (2) follows from the formula
\begin{equation}\label{action}
s_{\alpha+p\delta}(\beta+q\delta)=s_{\alpha}(\beta)+(q-(\beta,\alpha^{\vee})p)\delta
\end{equation}
and the fact $\pi(s_{\alpha+p\delta})=\pi(s_{\alpha}t_{-p\alpha^{\vee}})=s_{\alpha}$.
\end{proof}

\begin{lemma}\label{structure}
(1) $N(s_{\alpha+n\delta})$ has the following structure:

(i) Suppose that $\beta\in \Phi^+, s_{\alpha}(\beta)\in \Phi^+$. Then for $0\leq
k<\frac{2(\alpha,\beta)}{(\alpha,\alpha)}n$, $\beta+k\delta\in
N(s_{\alpha+n\delta}).$

(ii) Suppose that $\beta\in \Phi^+, s_{\alpha}(\beta)\in \Phi^-$. Then for
$0\leq k\leq\frac{2(\alpha,\beta)}{(\alpha,\alpha)}n$,
$\beta+k\delta\in N(s_{\alpha+n\delta}).$

(iii) Suppose that $\beta\in \Phi^-, s_{\alpha}(\beta)\in \Phi^+$. Then for
$0<k<\frac{2(\alpha,\beta)}{(\alpha,\alpha)}n$, $\beta+k\delta\in
N(s_{\alpha+n\delta}).$

(iv) Suppose that $\beta\in \Phi^-, s_{\alpha}(\beta)\in \Phi^-$. Then for
$0<k\leq\frac{2(\alpha,\beta)}{(\alpha,\alpha)}n$, $\beta+k\delta\in
N(s_{\alpha+n\delta}).$

(2) If there exists a  $\delta$-chain through $\beta$ beginning from $\beta_0$ in $N(s_{\alpha+n\delta})$  then there also exists a $\delta$-chain through $-s_{\alpha}(\beta)$ beginning from $(-s_{\alpha}(\beta))_0$ in $N(s_{\alpha+n\delta})$. In addition, these two chains are of equal length.
\end{lemma}

\begin{proof}
(1) follows from Equation \eqref{action}. To see (2), one notes that $(-s_{\alpha}(\beta),\alpha)=(s_{\alpha}(\beta),-\alpha)=(s_{\alpha}(\beta),s_{\alpha}(\alpha))=(\beta,\alpha)$ and then uses (1).
\end{proof}

\begin{lemma}
When $n$ is sufficiently large,
the inversion set $N(ws_{\alpha+n\delta})$ consists of some finite $\delta$-chains (through roots in $\Phi$) whose lengths are independent of $n$  and the following $\delta$-chains:
$$(\pi(w)(\alpha))_0, (\pi(w)(\alpha))_0+\delta,\cdots, (\pi(w)(\alpha))_0+k_0\delta,$$
$$(\beta_1)_0, (\beta_1)_0+\delta,\cdots, (\beta_1)_0+k_1\delta,$$
$$(\beta_2)_0, (\beta_2)_0+\delta,\cdots, (\beta_2)_0+k_2\delta,$$
$$\cdots$$
$$(\beta_t)_0, (\beta_t)_0+\delta,\cdots, (\beta_t)_0+k_t\delta$$
where $\beta_i\in \Phi, 1\leq i\leq t$ and one can pair each $\beta_i$ with some $\beta_j, j\neq i$ such that $p\beta_i+p\beta_j=\pi(w)(\alpha), p>0$.  Furthermore when $n$ is changed to $n+1$, the increment of the length of the $\delta$-chain through $\beta_i$  and that of the $\delta$-chain through $\beta_j$  coincides.
\end{lemma}

\begin{proof}
 By Lemma 4.1(e) of \cite{DyerWeakOrder} $$N(ws_{\alpha+n\delta})=(N(w)\backslash (-wN(s_{\alpha+n\delta})))\uplus (wN(s_{\alpha+n\delta})\backslash (-N(w))).$$

By Lemma \ref{dominance} (1) $N(w)$ consists some finite $\delta$-chains through certain roots in $\Phi$. Therefore we conclude that $(N(w)\backslash (-wN(s_{\alpha+n\delta})))$ consists of some finite $\delta$-chains (through roots in $\Phi$) whose lengths are independent of $n$ (when $n$ is sufficiently large).

By Lemma \ref{structure} (1) and Lemma \ref{dominance} (2), one sees that $wN(s_{\alpha+n\delta})\cap \widetilde{\Phi}^+=wN(s_{\alpha+n\delta})\backslash (-N(w))$ consists of $\delta$-chains through certain roots in $\Phi$ and that when $n$ grows the upper bound of the coefficient of $\delta$ in each of  these chains grows accordingly.

Again Lemma \ref{dominance} (1) forces that any  $\delta$-chain through a root $\beta$ must start from $\beta_0$ in $N(ws_{\alpha+n\delta})$.
Therefore we have shown that $N(ws_{\alpha+n\delta})$ consists of roots in the form as listed in the lemma.
 We still need to   show the existence of the pairing described in the Lemma.

Now suppose $\beta_i=\pi(w)(\beta)$. Then the root $\beta_j=\pi(w)(-s_{\alpha}(\beta))$ satisfies the condition in this lemma thanks to Lemma \ref{structure} (2). Note that the positivity of $p$ follows from the fact $(\beta,\alpha)>0$ (note that $(\beta,\alpha)>0$ because otherwise the $\delta$-chain through $\beta$ will not appear in the inversion set of $s_{\alpha+n\delta}$ by Lemma \ref{structure} (1)).
\end{proof}

\begin{lemma}\label{limit}
$$\lim_{n\rightarrow \infty} |l_B(s_{\alpha+n\delta}w)|=+\infty  $$
for any $B$ biclosed in $\widetilde{\Phi}^+$, $w\in \widetilde{W}, \alpha\in \Phi.$
\end{lemma}

\begin{proof}
We examine the set $N(ws_{\alpha+n\delta})$ for sufficiently large $n$.
By Lemma \ref{structure} besides some finite $\delta$-chain  which will stay stationary when $n$ grows, this set consists some finite $\delta$-chain through roots in $\Phi$ which will grow when $n$ grows, i.e.
$$(\pi(w)(\alpha))_0, (\pi(w)(\alpha))_0+\delta,\cdots, (\pi(w)(\alpha))_0+k_0\delta,$$
$$(\beta_1)_0, (\beta_1)_0+\delta,\cdots, (\beta_1)_0+k_1\delta,$$
$$(\beta_2)_0, (\beta_2)_0+\delta,\cdots, (\beta_2)_0+k_2\delta,$$
$$\cdots$$
$$(\beta_t)_0, (\beta_t)_0+\delta,\cdots, (\beta_t)_0+k_t\delta$$

Also one can pair those $\beta_i$'s, for each $\beta_i$ one can find $\beta_j, j\neq i$ such that $p\beta_i+p\beta_j=\pi(w)(\alpha), p>0$.
Furthermore when $n$ is changed to $n+1$ the increment of the length of the $\delta$-chain through $\beta_i$  and that of the $\delta$-chain through $\beta_j$  coincides.

Now we consider the set $I_B=\{\alpha\in \Phi|\{\alpha\}^{\wedge}\cap B$ is infinite$\}$. Then $I_B$ is a biclosed set in $\Phi.$
Also by Lemma 5.10 in \cite{DyerReflOrder} for any $\alpha\in I_B$, there exists $N_{\alpha}\in \mathbb{Z}_{>0}$ such that $\alpha+n\delta\in B$ for all $n>N_{\alpha}$.
Suppose that $\pi(w)(\alpha)\in I_B$. Then for a pair of $\beta_i,\beta_j$ either one of them is in $I_B$ and one of them is in $\Phi\backslash I_B$ or both of them are in $I_B$ since otherwise $p\beta_i+p\beta_j=\pi(w)(\alpha)$ has to be in $\Phi\backslash I_B$ by coclosedness of $I_B.$
So when $n$ is changed to $n+1$ more added roots in the inversion set $N(ws_{\alpha+n\delta})$ are in $B$ than in $\widehat{\Phi}\backslash B$, making the twisted length decreasing.

Similarly if $\pi(w)(\alpha)\not\in I_B$ one can deduce that when $n$ is changed to $n+1$ more added roots in the inversion set $N(ws_{\alpha+n\delta})$ are in $\widehat{\Phi}\backslash B$ than in $B$, making the length increasing. Therefore we see that the twisted length tends to infinity as $n\rightarrow \infty.$
\end{proof}

\begin{proposition}\label{finiteunder}
For any $x\in \widetilde{W}$, $n\in \mathbb{N}$ and a biclosed set $B$ in $\widetilde{\Phi}^+$, the set
$$\{y\in \widetilde{W}|y\leq_B x, l_B(x)-l_B(y)=n\}$$
is finite.
\end{proposition}

\begin{proof}
We prove this by induction. Let $n=1$ and $\alpha\in \Phi$. The set $\{k|l_B(s_{\alpha+k\delta}x)=l_B(x)-1\}$ is finite by Lemma \ref{limit}. Note that $\Phi$ is finite and therefore the proposition holds when $n=1.$ Now let $n=k$. Then
$$\{y\in \widetilde{W}|y\leq_B x, l_B(x)-l_B(y)=k\}=$$$$\bigcup_{\alpha,k:l_B(s_{\alpha+k\delta}x)=l_B(x)-1}\{y\in \widetilde{W}|y\leq_B s_{\alpha+k\delta}x, l_B(s_{\alpha+k\delta}x)-l_B(y)=k-1\}$$
since twisted Bruhat order is graded by twisted length function.
Now the proposition follows from induction.
\end{proof}

\begin{theorem}
For any biclosed set $B$ in $\widetilde{\Phi}^+$ and any $x,y\in \widetilde{W}$ such that $x<_B y$, the interval $[x,y]$ in the poset $(W,\leq_B)$ is finite.
\end{theorem}

\begin{proof}
If $l_B(y)-l_B(x)=1$ then the assertion is trivial (as $[x,y]=\{x,y\}$). Now suppose that the assertion holds for all pairs of $x,y$ with the properties $x<_B y$ and $l_B(y)-l_B(x)\leq k-1$. Take $x,y\in \widetilde{W}, x<_B y$ and $l_B(y)-l_B(x)=k$.
By Proposition \ref{finiteunder}, the set $\{z|x<_B z<_B y, l_B(y)-l_B(z)=1\}$ is finite. For any $z$ such that $x<_B z<_B y, l_B(y)-l_B(z)=1$ the interval $[x,z]$ is finite by induction assumption. Then $[x,y]$ is also finite since the twisted Bruhat order is graded by the left twisted length function.
\end{proof}

\section{Construction of infinite interval in a non-affine Coxeter system}\label{hyperinterval}

It is natural to ask that given a non-affine infinite Coxeter group $W$ whether infinite intervals always show up in $(W,\leq_B)$ for some biclosed set $B\subset \Phi^+.$ In \cite{DyerTwistedBruhat} subsection 1.10, it is shown that this is true for a universal Coxeter group of rank $>2$ as follow.

Let $(W,S)$ be a universal Coxeter system of rank $3$ with $S=\{s_1,s_2,s_3\}$. Let $B=N((s_1s_2s_3)^{\infty})$. Then the closed interval $[1,s_1s_2s_3]$ is infinite in $(W,\leq_B)$.

By \cite{UniversalRefl}, every irreducible infinite non-affine Coxeter group $W$ has a reflection subgroup $W'$ which is universal of rank $3$. Therefore by above for certain biclosed set $B\subset \Phi_{W'}^+$ (the set of positive roots of $W'$) the left $B-$twisted  strong Bruhat order on such a reflection subgroup (denoted $\leq_{W',B}$)  has an infinite interval (with bottom element $e$) which we denote by $[e,w]$.
Suppose that there exists a biclosed set $A\subset \Phi^+$ such that $A\cap \Phi_{W'}^+=B$. Then by Proposition 1.4 in \cite{DyerTwistedBruhat} (or see subsection 1.12 of \cite{quotient}) we have a poset injection $([e,w],\leq_{W',B})\rightarrow ([e,w],\leq_A)$.
 Therefore the closed interval $[e,w]$ is also infinite in $(W,\leq_A)$.

Now we illustrate this process through a concrete example.

Let $(W,S)$ be of $(2,3,\infty)$ type, i.e. $W=\{s_1, s_2, s_3|s_1^2=s_2^2=s_3^2=(s_1s_2)^3=(s_1s_3)^2=e\}$ and $S=\{s_1, s_2, s_3\}.$ Denote $\alpha_{s_i}$ by $\alpha_i$. For a reflection subgroup $W'$ of $W$, denote $T\cap W'$ by $T_{W'}$. The root system of $W'$ is $\Phi_{W'}=\{\alpha\in \Phi|s_{\alpha}\in W'\}.$

Consider the reflection subgroup $W'$ generated by $S':=\{s_1, s_2s_3s_2, s_3s_2s_3s_2s_3\}$. Write $r_1=s_1, r_2=s_2s_3s_2, r_3=s_3s_2s_3s_2s_3.$ Note that $(\alpha_{r_i},\alpha_{r_j})=-1$ for $i,j\in \{1,2,3\},i\neq j.$
Also we note that $$\widetilde{N}(s_1)\backslash \{s_1\}=\emptyset,$$ $$\widetilde{N}(s_2s_3s_2)\backslash \{s_2s_3s_2\}=\{s_2,s_2s_3s_2s_3s_2\},$$ $$\widetilde{N}(s_3s_2s_3s_2s_3)\backslash\{s_3s_2s_3s_2s_3\}=\{s_3,s_3s_2s_3, s_3s_2s_3s_2s_3s_2s_3, s_3s_2s_3s_2s_3s_2s_3s_2s_3\}.$$
One easily checks that  $r_i$ and $r_j, i\neq j, i,j\in \{1,2,3\}$ generate an infinite dihedral subgroup with $r_i$ and $r_j$ being the canonical simple generators.
Then apply Proposition 3.5 of \cite{DyerReflSubgrp} we see that $\widetilde{N}(s_1)\cap T_{W'}=\{s_1\}, \widetilde{N}(s_2s_3s_2)\cap T_{W'}=\{s_2s_3s_2\},$ $\widetilde{N}(s_3s_2s_3s_2s_3)\cap T_{W'}=\{s_3s_2s_3s_2s_3\}.$
Therefore $(W',S')$ is indeed a rank 3 universal Coxeter system.

Since the only braid moves that can be applied to $(s_1(s_2s_3s_2)(s_3s_2s_3s_2s_3))^k$ are short braid moves between $s_1$ and $s_3$, the element $s_1(s_2s_3s_2)(s_3s_2s_3s_2s_3)$ is straight. One can also check that $\widetilde{N}(s_1(s_2s_3s_2)(s_3s_2s_3s_2s_3))\cap T_{W'}=\{r_1, r_1r_2r_1, r_1r_2r_3r_2r_1\}$. Consequently $\widetilde{N}((s_1(s_2s_3s_2)(s_3s_2s_3s_2s_3))^{\infty})\cap T_{W'}=\newline$
$\{r_1,r_1r_2r_1,r_1r_2r_3r_2r_1, r_1r_2r_3r_1r_3r_2r_1,r_1r_2r_3r_1r_2r_1r_3r_2r_1\cdots\}$. Based on the above discussion, one concludes that the interval $[1, s_1(s_2s_3s_2)(s_3s_2s_3s_2s_3)]$ in $(W,\leq_A)$ is infinite where $A=N((s_1(s_2s_3s_2)(s_3s_2s_3s_2s_3))^{\infty})$.

\section{Set of elements of a given twisted length}\label{fixlength}

In this section we consider the set $\{w\in \widetilde{W}|l_B'(w)=n\}$. In contrast to the situation for the ordinary Bruhat order, we will show that such a set is infinite in most cases. Since $l_B(w^{-1})=l_B'(w)$, our result also shows that the set $\{w\in \widetilde{W}|\,l_B(w)=n\}$ is infinite.

It is convenient to investigate this question in the context of (right) $B$-twisted weak Bruhat order $<_B'$. In the following two lemmas $(W,S)$ is a general finite rank infinite Coxeter system and $\Phi,\Phi^+$ are its root system and its set of positive roots.

\begin{lemma}\label{nomaxminone}
Let $w\in W_l$ (the set of  infinite reduced words), then $(W,\leq_{N(w)}')$ (resp. $(W,\leq_{\Phi^+\backslash N(w)}')$) has no maximal element or minimal element.
\end{lemma}

\begin{proof}
Suppose to the contrary that $u$ is a maximal element of $(W,\leq_{N(w)}')$ (i.e. for any element $p\in W$ either $u$ and $p$ are not comparable or $u\geq'_{N(w)} p$). Let $S=\{s_1,s_2,\cdots,s_n\}$ be the set of simple reflections. Then $us_i\lhd u$ for all $s_i\in S$. Let $U_i=N(us_i)\backslash N(u)$. These sets are contained in $N(w)$ by definition. Therefore there exists a (finite) prefix $v$ of $w$ such that $U_i\subset N(v)$ for all $i$. Also note that $N(u)\backslash N(us_i)\subset \Phi^+\backslash N(w)\subset \Phi^+\backslash N(v)$. Hence under the  twisted weak order $(W,\leq_{N(v)}')$, one has $us_i\lhd u$ for all $i$. Therefore $u$ is a maximal element in $(W,\leq_{N(v)}')$ since if there exists some $u'\geq u$ then by the chain property of the twisted weak order there must be some $s_i$ such that $us_i\vartriangleright u.$ But  $\leq_{N(v)}'$ is isomorphic to the standard right weak order which has no maximal element when $W$ is infinite. A contradiction. Therefore we have shown that $(W,\leq_{N(w)}')$ has no maximal elements.

Now we claim that it cannot happen that there exists $u\in W$ such that $u\leq_{N(w)}' x, \forall x\in W$. Since $N(u)$ is finite, there exists $\alpha\in \Phi^+$ such that $\alpha\in N(w)\backslash N(u)$. Then $N(s_{\alpha})\backslash N(u)\not\subset (\Phi^+\backslash N(w))$. So $u\nleq_{N(w)}' s_{\alpha}$ and we establish the claim. Now for any $u\in W$ we take $x\in W$ such that $u\not\leq_{N(w)}' x$, and then consider the meet of $u$ and $x$ (it exists since $(W,\leq_{N(w)}')$ is a meet semilattice). Such a meet is clearly not equal to $u$ as $u\not\leq_{N(w)}' x$. Therefore $(W,\leq_{N(w)}')$ has no minimal element.

The dual assertion about $(W,\leq_{\Phi^+\backslash N(w)}')$ can be proved in the same fashion and thus is omitted.
\end{proof}

\begin{lemma}\label{dotactioniso}
Let $B$ be a biclosed set in $\Phi^+$ and $w\in W$. Then as posets $(W,\leq_B')\simeq (W,\leq_{w\cdot B}')$ under the map $u\mapsto wu.$
\end{lemma}

\begin{proof}
Suppose that $u\leq_B'v$. We show that $N(wu)\backslash N(wv)$ is contained in $w\cdot B=(N(w)\backslash -w(B))\uplus (w(B)\backslash -N(w))$.

Note that $N(wu)\backslash N(wv)=((N(w)\backslash -wN(u))\backslash (N(w)\backslash -wN(v)))\uplus ((wN(u)\backslash -N(w)\backslash (w(N(v))\backslash -N(w)))).$
Take $\alpha\in (N(w)\backslash -wN(u))\backslash (N(w)\backslash -wN(v))$. This implies that $\alpha\in N(w)$ and $-w^{-1}(\alpha)\in N(v)\backslash N(u)$. Since $N(v)\subset N(u)\subset \Phi^+\backslash B$, $\alpha\not\in -w(B).$
So $\alpha\in N(w)\backslash -w(B).$
Take $\alpha\in (wN(u)\backslash -N(w)\backslash (w(N(v))\backslash -N(w)))$. This implies that $\alpha\not\in -N(w)$ and $\alpha\in w(N(u)\backslash N(v))$. Therefore $\alpha\in w(B)\backslash -N(w).$
Therefore we see that $N(wu)\backslash N(wv)\subset w\cdot B.$

Next we show that $N(wv)\backslash N(wu)\subset \Phi^+\backslash w\cdot B$. Since $v\leq_{\Phi^+\backslash B}'u$, the above argument shows that
$N(wv)\backslash N(wu)\subset w\cdot (\Phi^+\backslash B)$. It follows from the formula $w\cdot B:=(N(w)\backslash w(-B))\cup (w(B)\backslash
(-N(w)))$ that $w\cdot (\Phi^+\backslash B)=\Phi^+\backslash w\cdot B.$

Therefore we conclude that $wu\leq_{w\cdot B}' wv$.

The above arguments also prove that the inverse map $u\mapsto w^{-1}u$  preserves the order since $w^{-1}\cdot(w\cdot B)=B$ and hence the lemma is proved.
\end{proof}

Now let $\Phi,\Psi^+$ and $\Delta$ be a finite crystallographic root system, a positive system of it and the corresponding simple system respectively. Suppose that $\Delta_1\subset \Delta$. Let $\Phi_{\Delta_1}=\mathbb{R}\Delta_1\cap \Phi$ be the root subsystem generated by $\Delta_1$. Then $\mathbb{R}_{\geq 0}\Delta_1\cap \Phi=\Phi_{\Delta_1}^+$ is a positive system of $\Phi_{\Delta_1}$.
Let $W$ be the finite Weyl group with the root system $\Phi$ and $\widetilde{W}$ be  the corresponding affine Weyl group. A subset $\Gamma$ of $\Phi$ is said to be $\mathbb{Z}$-closed if for $\alpha,\beta\in \Gamma$ such that $\alpha+\beta\in \Phi$ one has $\alpha+\beta\in \Gamma.$

 \begin{lemma}\label{assemble}
 Let $\Gamma^+$ be another positive system of $\Phi_{\Delta_1}$. Then $(\Psi^+\backslash \Phi_{\Delta_1}^+)\cup \Gamma^+$ is a positive system of $\Phi$. If $\Delta_2\subset \Delta$ such that $(\Delta_1,\Delta_2)=0$, $\Delta_2$ is also a subset of the simple system of $(\Psi^+\backslash \Phi_{\Delta_1}^+)\cup \Gamma^+$.
\end{lemma}

\begin{proof}
Find the element $w$ in the parabolic subgroup generated by $s_{\alpha},\alpha\in \Delta_1$ such that $w\Phi_{\Delta_1}^+=\Gamma^+$. Act $w$ on $\Psi^+$ and we get $(\Psi^+\backslash \Phi_{\Delta_1}^+)\cup \Gamma^+$. Note that we have $w(\Delta_2)=\Delta_2.$
\end{proof}

Let $B$ be a biclosed set in $\widetilde{\Phi}^+$.

\begin{lemma}\label{nomaxmintwo}
If $B$ and $\widetilde{\Phi}^+\backslash B$ are both infinite, then $(\widetilde{W},\leq_B')$ has no maximal element or minimal element.
\end{lemma}

\begin{proof}
By the characterization of the biclosed sets in $\widetilde{\Phi}$, $I_B=P(\Psi^+,\Delta_1,\Delta_2)$ for some positive system $\Psi^+$ and some orthogonal subsets $\Delta_1, \Delta_2$ of its simple system.
By Lemma \ref{nomaxminone} and Theorem \ref{infiniteword}, it suffices to treat the case where $I_B=P(\Psi^+,\Delta_1,\Delta_2)$ where $\Delta_1$ and $\Delta_2$ are both nonempty. Further it suffices to check the case where $B=P(\Psi^+,\Delta_1,\Delta_2)^{\wedge}$ as all $(W,\leq_B')$ with $I_B=P(\Psi^+,\Delta_1,\Delta_2)$  are isomorphic by Lemma \ref{dotactioniso}. We suppose to the contrary that there exists $u\in W$ such that  $us_i\rhd u$ for all $s_i$ simple reflection of $\widetilde{W}$ (so $u$ is minimal by the chain property of the twisted weak order).

We write  $A$ for $A_{N(u)}$. It is $\mathbb{Z}-$closed in $\Phi$ by Lemma 3.13 in  \cite{wang}. We also have $A\cap -A=\emptyset$ by subsection \ref{biclosedaffine}.  Then we deduce that $A':=\Phi_{\Delta_2}\cap A$ is $\mathbb{Z}$-closed in $\Phi_{\Delta_2}$
and that $A'\cap -A'=\emptyset$. Then by \cite{Bourbaki} Chapter IV, $\S 1$, Proposition 22, $A'$ is contained in a positive system $\Gamma^+$ in $\Phi_{\Delta_2}.$ Let $\mathbb{R}_{\geq 0}\Delta_2\cap \Phi=\Phi_{\Delta_2}^+.$
Then  $A'$ is contained in $(\Psi^+\backslash \Phi_{\Delta_2}^+\cup \Gamma^+)\backslash \Phi_{\Delta_1}$ where $\Psi^+\backslash \Phi_{\Delta_2}^+\cup \Gamma^+$ is a positive system in $\Phi$ by Lemma \ref{assemble}. We denote this positive system by $\Xi^+$. Also by Lemma \ref{assemble} $\Delta_1$ is a subset of $\Xi^+$'s simple system.
Since $A_{N(u)\backslash N(us_i)}\cap \Phi_{\Delta_2}\subset A_{N(u)}\cap \Phi_{\Delta_2}(=A')$, so
\begin{equation}\label{partone}
A_{N(u)\backslash N(us_i)}\cap \Phi_{\Delta_2}\subset \Xi^+\backslash \Phi_{\Delta_1}.
\end{equation}

Because $N(u)\backslash N(us_i)\subset B$, $A_{N(u)\backslash N(us_i)}\subset P(\Psi^+,\Delta_1,\Delta_2)=(\Psi^+\backslash \Phi_{\Delta_1})\cup \Phi_{\Delta_2}$. Hence
\begin{equation}\label{parttwo}
A_{N(u)\backslash N(us_i)}\backslash \Phi_{\Delta_2}\subset (\Psi^+\backslash \Phi_{\Delta_1})\backslash \Phi_{\Delta_2}\subset (\Psi^+\backslash \Phi_{\Delta_2}^+\cup \Gamma^+)\backslash \Phi_{\Delta_1}=\Xi^+\backslash \Phi_{\Delta_1}.
 \end{equation}
 Hence by equation \ref{partone} and \ref{parttwo} we conclude that $A_{N(u)\backslash N(us_i)}\subset \Xi^+\backslash \Phi_{\Delta_1}$. Consequently $N(u)\backslash N(us_i)\subset (\Xi^+\backslash \Phi_{\Delta_1})^{\wedge}$ for all $i$. The right hand side is the inversion set of an infinite reduced word $w$ by Theorem \ref{infiniteword}.

On the other hand, one has the inclusion $N(us_i)\backslash N(u)\subset \widetilde{\Phi}^+\backslash P(\Psi^+,\Delta_1,\Delta_2)^{\wedge}\subset \widetilde{\Phi}^+\backslash (\Xi^+\backslash \Phi_{\Delta_1})^{\wedge}$. Thus we conclude $u$ is minimal under $(W,\leq_{N(w)}')$. This contradicts Lemma \ref{nomaxminone}.

Similarly one can also prove that the poset has no maximal element.
\end{proof}

\begin{theorem}\label{setisinfinite}
Suppose that $\widetilde{W}$ is irreducible and of rank $\geq 3.$ Let $k\in \mathbb{Z}$ and let $B$ be an infinite biclosed set in $\widetilde{\Phi}^+$ such that $\widetilde{\Phi}^+\backslash B$ is also infinite. The set $\{w\in \widetilde{W}|\,l_B'(w)=k\}$ is infinite.
\end{theorem}

\begin{proof}
Suppose that to the contrary there exists $k\in \mathbb{Z}$, the set $\{w\in \widetilde{W}|\,l_B'(w)=k\}$ is finite. Suppose that this set is $\{w_1,w_2,\cdots,w_t\}$. We show that under the right $B$-twisted weak order $\leq_B'$ any element is either greater than one of $w_i$ or less than one of $w_i$ (depending whether its twisted length is $\geq$ or $\leq k$).

 Let $u$ be an element in $W$ and assume  without loss of generality that $l_B'(u)>k.$ By Lemma \ref{nomaxmintwo} $u$ is not minimal. So one can find $s\in S$ such that $us<_B'u$ and $l_B'(us)=l_B'(u)-1$. Proceed this way one sees that $u$ must be greater than some element with twisted length $k$.

 Now by Lemma \ref{dotactioniso} we only need to consider the case where $B=P(\Psi^+,\Delta_1,\Delta_2)^{\wedge}$. If $\Delta_1=\Delta$, $B$ is empty.  Let $\alpha\in \Delta\backslash \Delta_1$ and $\rho$ be the highest root in $\Psi^+$. Since $\cup_{i}N(w_i)$ is finite, there must exist some positive roots $\alpha+p\delta\not\in \cup_{i}N(w_i)$ and $-\rho+q
\delta\not\in \cup_{i}N(w_i)$. We claim that there exists some $v$ such that $N(v)$ contains $-\rho+q\delta, \alpha+p\delta.$

Proof of the claim. Since the rank of $\Phi$ is at least two, $\alpha\neq \rho$. So $-\rho,\alpha\in \Omega^+:=\Psi^-\backslash \{-\alpha\}\cup \{\alpha\}$, which is another positive system. So $-\rho+q\delta, \alpha+p\delta$ are contained in $(\Omega^+)^{\wedge}$, which is the inversion set of an infinite reduced word. Consequently they are contained in the inversion set of a finite prefix $v$ of this infinite reduced word.

Note that $\Delta_2\neq \Delta$ since that will cause $B=\widetilde{\Phi}^+=\Phi^{\wedge}$ and then $B$ is cofinite. Therefore $-\rho+q\delta\not\in B$.
However $\alpha+p\delta\in B$.
Therefore $v$ is not comparable to any of $w_i$ under the right $B$-twisted weak order. Contradiction.
\end{proof}

\begin{corollary}
Suppose that $\widetilde{W}$ is irreducible and of rank $\geq 3.$ Let $B$ be an infinite biclosed set in $\widetilde{\Phi}^+$ such that $\widetilde{\Phi}^+\backslash B$ is also infinite. Then the $B$-twisted weak Bruhat order $(\widetilde{W},\leq_B')$ (resp. the $B$-twisted strong Bruhat order $(\widetilde{W},\leq_B)$) has an infinite antichain.
\end{corollary}

\begin{proof}
Note that if $u\lneq'_B v$ (resp. $u\lneq_B v$) then $l_B'(v)>l_B'(u)$ (resp. $l_B(v)>l_B(u)$). Therefore the set $\{w\in \widetilde{W}|\,l_B'(w)=k\}$ (resp. $\{w\in \widetilde{W}|\,l_B(w)=k\}$) forms an infinite antichain by Theorem \ref{setisinfinite}.
\end{proof}

We remark that in \cite{antichain} it is proved that if $B$ is finite or cofinite, the $B$-twisted weak order (and consequently the $B$-twisted strong Bruhat order) on an affine Weyl group does not have an infinite antichain.

\section{An example: alcove order of $\widetilde{A}_2$}\label{examplesec}

In this section we provide a clear picture of a specific twisted (strong) Bruhat order.
Let $W$ and $\Phi$ be of type $A_2$. The two simple roots are denoted by $\alpha, \beta$.
We consider the $B$-twisted (strong) Bruhat order $(\widetilde{W}, \leq_B)$ where $B=\{\alpha,\alpha+\beta,\beta\}^{\wedge}$. Such an order is also called alcove order in \cite{quotient}. We describe all of its spherical intervals.

In the following lemma we characterize all covering relations in $(\widetilde{W}, \leq_B)$.
\begin{lemma}\label{sixtype}
For $(\widetilde{W},\leq_B)$ one has the following results:

(1) Suppose that $w\in T$.
Then $l_B(s_{\alpha+k\delta}w)=l_B(w)-2k-1, l_B(s_{\beta+k\delta}w)=l_B(w)-2k-1, l_B(s_{\alpha+\beta+k\delta}w)=l_B(w)-4k-3.$

(2) Suppose that $w\in s_{\alpha}s_{\beta}T$.
Then $l_B(s_{\alpha+k\delta}w)=l_B(w)+4k+1, l_B(s_{\beta+k\delta}w)=l_B(w)-2k-1, l_B(s_{\alpha+\beta+k\delta}w)=l_B(w)+2k+1.$

(3) Suppose that $w\in s_{\beta}s_{\alpha}T$.
Then $l_B(s_{\alpha+k\delta}w)=l_B(w)-2k-1, l_B(s_{\alpha+\beta+k\delta}w)=l_B(w)+2k+1, l_B(s_{\beta+k\delta}w)=l_B(w)+4k+1.$

(4) Suppose that $w\in s_{\beta}s_{\alpha}s_{\beta}T$.
Then $l_B(s_{\alpha+k\delta}w)=l_B(w)+2k+1, l_B(s_{\beta+k\delta}w)=l_B(w)+2k+1, l_B(s_{\alpha+\beta+k\delta}w)=l_B(w)+4k+3.$

(5) Suppose that $w\in s_{\beta}T$.
Then $l_B(s_{\alpha+k\delta}w)=l_B(w)-4k-1, l_B(s_{\beta+k\delta}w)=l_B(w)+2k+1, l_B(s_{\alpha+\beta+k\delta}w)=l_B(w)-2k-1.$

(6) Suppose that  $w\in s_{\alpha}T$.
Then $l_B(s_{\alpha+k\delta}w)=l_B(w)+2k+1, l_B(s_{\alpha+\beta+k\delta}w)=l_B(w)-2k-1, l_B(s_{\beta+k\delta}w)=l_B(w)-4k-1.$

\end{lemma}

\begin{proof}
It follows from the fact $\widetilde{W}=W\ltimes T$  that every element is in one of the six subsets and these six subsets are pairwise disjoint.

One can verify by induction easily that
$$N(s_{\alpha+k\delta})=\alpha_0^{2k}(-\beta)_1^k(\alpha+\beta)_0^{k-1}, k\geq 0,$$
$$N(s_{\alpha+k\delta})=(-\alpha)_1^{-2k-1}(\beta)_0^{-k-1}(-\alpha-\beta)_1^{-k}, k<0,$$
$$N(s_{\alpha+\beta+k\delta})=(\alpha+\beta)_0^{2k}\alpha_0^k\beta_0^k, k\geq 0,$$
$$N(s_{\alpha+\beta+k\delta})=(-\alpha-\beta)_1^{-2k-1}(-\alpha)_1^{-k-1}(-\beta)_1^{-k-1}, k<0,$$
$$N(s_{\beta+k\delta})=\beta_0^{2k}(-\alpha)_1^k(\alpha+\beta)_0^{k-1}, k\geq 0,$$
$$N(s_{\beta+k\delta})=(-\beta)_1^{-2k-1}(\alpha)_0^{-k-1}(-\alpha-\beta)_1^{-k}, k<0.$$
Now by using the identity $N(wu)=(N(w)\backslash w(-N(u)))\cup (wN(u)\backslash
(-N(w)))$, we compute
$$N(t_{k_1\alpha+k_2\beta})$$
$$=(\alpha_0^{k_1-\frac{k_2}{2}-1}(-\alpha)_1^{-k_1+\frac{k_2}{2}}\beta_0^{-\frac{k_1}{2}+k_2-1}(-\beta)_1^{\frac{k_1}{2}-k_2}
(\alpha+\beta)_0^{\frac{k_1}{2}+\frac{k_2}{2}-1}(-\alpha-\beta)_1^{-\frac{k_1}{2}-\frac{k_2}{2}}).$$
$$N(t_{k_1\alpha+k_2\beta}s_{\alpha+k\delta})$$
$$=\alpha_{0}^{k_1-\frac{k_2}{2}+2k}(-\alpha)_1^{-2k-1-k_1+\frac{k_2}{2}}\beta_0^{-k-1-\frac{k_1}{2}+k_2}(-\beta)_1^{k+\frac{k_1}{2}-k_2}$$
$$(\alpha+\beta)_0^{k-1+\frac{k_1+k_2}{2}}(-\alpha-\beta)_1^{-k-\frac{k_1+k_2}{2}}.$$
$l_B(s_{\alpha+k\delta}t_{k_1\alpha+k_2\beta})=-2k-1-k_1+\frac{k_2}{2}+k+\frac{k_1}{2}-k_2-k-\frac{k_1+k_2}{2}=-2k-k_1-k_2-1.$
$l_B(t_{k_1\alpha+k_2\beta})=-k_1+\frac{k_2}{2}+\frac{k_1}{2}-k_2-\frac{k_1}{2}-\frac{k_2}{2}=-k_1-k_2$.
Hence $l_B(t_{k_1\alpha+k_2\beta})-2k-1=l_B(s_{\alpha+k\delta}t_{k_1\alpha+k_2\beta})$.

The map $\tau: s_{\alpha}\mapsto s_{\beta}, s_{\beta}\mapsto s_{\alpha}, s_{\delta-\alpha-\beta}\mapsto s_{\delta-\alpha-\beta}$ induces an automorphism of $\widetilde{W}$.
From this and the above calculation, one sees that
 $l_B(t_{k_1\alpha+k_2\beta})-2k-1=l_B(t_{k_1\alpha+k_2\beta}s_{\beta+k\delta})$.
 We compute
$$N(t_{k_1\alpha+k_2\beta}s_{\alpha+\beta+k\delta})$$
 $$=\alpha_0^{k+k_1-\frac{k_2}{2}}(-\alpha)_1^{-k-1-k_1+\frac{k_2}{2}}\beta_0^{k-\frac{k_1}{2}+k_2}(-\beta)_1^{-k-1+\frac{k_1}{2}-k_2}$$
$$(\alpha+\beta)_0^{2k+\frac{k_1+k_2}{2}}(-\alpha-\beta)_1^{-2k-1-\frac{k_1+k_2}{2}}.$$
$l_B(s_{\alpha+\beta+k\delta}t_{k_1\alpha+k_2\beta})=-k-1-k_1+\frac{k_2}{2}-k-1+\frac{k_1}{2}-k_2-2k-1-\frac{k_1+k_2}{2}=-4k-k_1-k_2-3.$
Hence $l_B(s_{\alpha+\beta+k\delta}t_{k_1\alpha+k_2\beta})=l_B(t_{k_1\alpha+k_2\beta})-4k-3.$ So we have proved (1).

We compute
$$N(t_{k_1\alpha+k_2\beta}s_{\alpha})$$
$$=(\alpha_0^{k_1-\frac{k_2}{2}}(-\alpha)_1^{-k_1+\frac{k_2}{2}-1}\beta_0^{-\frac{k_1}{2}+k_2-1}(-\beta)_1^{\frac{k_1}{2}-k_2}
(\alpha+\beta)_0^{\frac{k_1}{2}+\frac{k_2}{2}-1}(-\alpha-\beta)_1^{-\frac{k_1}{2}-\frac{k_2}{2}}).$$
$$N(t_{k_1\alpha+k_2\beta}s_{\alpha}s_{\alpha+k\delta})$$
$$=\alpha_0^{-2k-1+k_1-\frac{k_2}{2}}(-\alpha)_1^{2k-k_1+\frac{k_2}{2}}\beta_0^{k-1-\frac{k_1}{2}+k_2}(-\beta)_1^{-k+\frac{k_1}{2}-k_2}$$
$$(\alpha+\beta)_0^{-k-1+\frac{k_1+k_2}{2}}(-\alpha-\beta)_1^{k-\frac{k_1+k_2}{2}}.$$
So $l_B(s_{\alpha+k\delta}s_{\alpha}t_{k_1\alpha+k_2\beta})=2k-k_1+\frac{k_2}{2}-k+\frac{k_1}{2}-k_2+k-\frac{k_1+k_2}{2}=2k-k_1-k_2$.
$l_B(s_{\alpha}t_{k_1\alpha+k_2\beta})=-k_1+\frac{k_2}{2}-1+\frac{k_1}{2}-k_2-\frac{k_1}{2}-\frac{k_2}{2}=-k_1-k_2-1$.
Therefore
$l_B(s_{\alpha+k\delta}s_{\alpha}t_{k_1\alpha+k_2\beta})=l_B(s_{\alpha}t_{k_1\alpha+k_2\beta})+2k+1.$
$$N(t_{k_1\alpha+k_2\beta}s_{\alpha}s_{\beta+k\delta})$$
$$=\alpha_0^{k+k_1-\frac{k_2}{2}}(-\alpha)_1^{-k-1-k_1+\frac{k_2}{2}}(\beta)_0^{k-1-\frac{k_1}{2}+k_2}(-\beta)_1^{-k+\frac{k_1}{2}-k_2}$$
$$(\alpha+\beta)_0^{2k+\frac{k_1+k_2}{2}}(-\alpha-\beta)_1^{-2k-1-\frac{k_1+k_2}{2}}.$$
So $l_B(s_{\beta+k\delta}s_{\alpha}t_{k_1\alpha+k_2\beta})=-k-1-k_1+\frac{k_2}{2}-k+\frac{k_1}{2}-k_2-2k-1-\frac{k_1+k_2}{2}=-4k-2-k_1-k_2=l_B(s_{\alpha}t_{k_1\alpha+k_2\beta})-4k-1.$
$$N(t_{k_1\alpha+k_2\beta}s_{\alpha}s_{\alpha+\beta+k\delta})$$
$$=\alpha_0^{-k-1+k_1-\frac{k_2}{2}}(-\alpha)_1^{k-k_1+\frac{k_2}{2}}\beta_0^{2k-\frac{k_1}{2}+k_2}(-\beta)_1^{-2k-1+\frac{k_1}{2}-k_2}$$
$$(\alpha+\beta)_0^{k+\frac{k_1+k_2}{2}}(-\alpha-\beta)_1^{-k-1-\frac{k_1+k_2}{2}}.$$
So $l_B(s_{\alpha+\beta+k\delta}s_{\alpha}t_{k_1\alpha+k_2\beta})=k-k_1+\frac{k_2}{2}-2k-1+\frac{k_1}{2}-k_2-k-1-\frac{k_1+k_2}{2}=-2k-k_1-k_2-2=l_B(s_{\alpha}t_{k_1\alpha+k_2\beta})-2k-1.$
So we have proved (6).

Using the map $\tau$ and the calculation for the case (6), one sees  (5).

Now we prove (3). We compute
$$N(t_{k_1\alpha+k_2\beta}s_{\alpha}s_{\beta})$$
$$=(\alpha_0^{k_1-\frac{k_2}{2}}(-\alpha)_1^{-k_1+\frac{k_2}{2}-1}\beta_0^{-\frac{k_1}{2}+k_2-1}(-\beta)_1^{\frac{k_1}{2}-k_2}
(\alpha+\beta)_0^{\frac{k_1}{2}+\frac{k_2}{2}}(-\alpha-\beta)_1^{-\frac{k_1}{2}-\frac{k_2}{2}-1}).$$
$$N(t_{k_1\alpha+k_2\beta}s_{\alpha}s_{\beta}s_{\alpha+k\delta})$$
$$=\alpha_0^{-k+k_1-\frac{k_2}{2}}(-\alpha)_1^{k-1-k_1+\frac{k_2}{2}}\beta_{0}^{2k-\frac{k_1}{2}+k_2}(-\beta)_1^{-2k-1+\frac{k_1}{2}-k_2}$$
$$(\alpha+\beta)_0^{k+\frac{k_1+k_2}{2}}(-\alpha-\beta)_1^{-k-1-\frac{k_1+k_2}{2}}.$$
So $l_B(s_{\alpha+k\delta}s_{\beta}s_{\alpha}t_{k_1\alpha+k_2\beta})=k-1-k_1+\frac{k_2}{2}-2k-1+\frac{k_1}{2}-k_2-k-1-\frac{k_1+k_2}{2}=-2k-k_1-k_2-3.$
$l_B(s_{\beta}s_{\alpha}t_{k_1\alpha+k_2\beta})=-k_1+\frac{k_2}{2}-1+\frac{k_1}{2}-k_2-\frac{k_1}{2}-\frac{k_2}{2}-1=-k_1-k_2-2$. Therefore
$l_B(s_{\alpha+k\delta}s_{\beta}s_{\alpha}t_{k_1\alpha+k_2\beta})=l_B(s_{\beta}s_{\alpha}t_{k_1\alpha+k_2\beta})-2k-1.$

We compute
$$N(t_{k_1\alpha+k_2\beta}s_{\alpha}s_{\beta}s_{\beta+k\delta})$$
$$=\alpha_0^{-k+k_1-\frac{k_2}{2}}(-\alpha)_1^{k-1-k_1+\frac{k_2}{2}}\beta_0^{-k-1-\frac{k_1}{2}+k_2}(-\beta)_1^{k+\frac{k_1}{2}-k_2}$$
$$(\alpha+\beta)_0^{-2k-1+\frac{k_1+k_2}{2}}(-\alpha-\beta)_1^{2k-\frac{k_1+k_2}{2}}.$$
So $l_B(s_{\beta+k\delta}s_{\beta}s_{\alpha}t_{k_1\alpha+k_2\beta})=k-1-k_1+\frac{k_2}{2}+k+\frac{k_1}{2}-k_2+2k-\frac{k_1+k_2}{2}=4k-k_1-k_2-1.$
Therefore $l_B(s_{\beta+k\delta}s_{\beta}s_{\alpha}t_{k_1\alpha+k_2\beta})=l_B(s_{\beta}s_{\alpha}t_{k_1\alpha+k_2\beta})+4k+1$.

We compute
$$N(t_{k_1\alpha+k_2\beta}s_{\alpha}s_{\beta}s_{\alpha+\beta+k\delta})$$
$$=\alpha_0^{-2k-1+k_1-\frac{k_2}{2}}(-\alpha)_1^{2k-k_1+\frac{k_2}{2}}\beta_0^{k-\frac{k_1}{2}+k_2}(-\beta)_1^{-k-1+\frac{k_1}{2}-k_2}$$
$$(\alpha+\beta)_0^{-k-1+\frac{k_1+k_2}{2}}(-\alpha-\beta)_1^{k-\frac{k_1+k_2}{2}}.$$
So $l_B(s_{\alpha+\beta+k\delta}s_{\beta}s_{\alpha}t_{k_1\alpha+k_2\beta})=2k-k_1+\frac{k_2}{2}-k-1+\frac{k_1}{2}-k_2+k-\frac{k_1+k_2}{2}=2k-k_1-k_2-1.$
Therefore $l_B(s_{\alpha+\beta+k\delta}s_{\beta}s_{\alpha}t_{k_1\alpha+k_2\beta})=l_B(s_{\beta}s_{\alpha}t_{k_1\alpha+k_2\beta})+2k+1$.

Using the map $\tau$ and the calculation for the $s_{\beta}s_{\alpha}T$ case, one sees  (2).

Now we prove (4). We compute
$$N(t_{k_1\alpha+k_2\beta}s_{\alpha}s_{\beta}s_{\alpha})$$
$$=(\alpha_0^{k_1-\frac{k_2}{2}}(-\alpha)_1^{-k_1+\frac{k_2}{2}-1}\beta_0^{-\frac{k_1}{2}+k_2}(-\beta)_1^{\frac{k_1}{2}-k_2-1}
(\alpha+\beta)_0^{\frac{k_1}{2}+\frac{k_2}{2}}(-\alpha-\beta)_1^{-\frac{k_1}{2}-\frac{k_2}{2}-1}).$$
$$N(t_{k_1\alpha+k_2\beta}s_{\alpha}s_{\beta}s_{\alpha}s_{\alpha+k\delta})$$
$$=\alpha_0^{k+k_1-\frac{k_2}{2}}(-\alpha)_1^{-k-1-k_1+\frac{k_2}{2}}\beta_0^{-2k-1-\frac{k_1}{2}+k_2}(-\beta)_1^{2k+\frac{k_1}{2}-k_2}$$
$$(\alpha+\beta)_0^{-k+\frac{k_1+k_2}{2}}(-\alpha-\beta)_1^{k-1-\frac{k_1+k_2}{2}}.$$
So $l_B(s_{\alpha+k\delta}s_{\alpha}s_{\beta}s_{\alpha}t_{k_1\alpha+k_2\beta})=-k-1-k_1+\frac{k_2}{2}+2k+\frac{k_1}{2}-k_2+k-1-\frac{k_1+k_2}{2}=2k-2-k_1-k_2.$
$l_B(s_{\alpha}s_{\beta}s_{\alpha}t_{k_1\alpha+k_2\beta})=-k_1+\frac{k_2}{2}-1+\frac{k_1}{2}-k_2-1-\frac{k_1}{2}-\frac{k_2}{2}-1=-k_1-k_2-3.$
Therefore $l_B(s_{\alpha+k\delta}s_{\alpha}s_{\beta}s_{\alpha}t_{k_1\alpha+k_2\beta})=l_B(s_{\alpha}s_{\beta}s_{\alpha}t_{k_1\alpha+k_2\beta})+2k+1.$

Using the map $\tau$ and the above calculation, one sees that $l_B(s_{\beta+k\delta}s_{\alpha}s_{\beta}s_{\alpha}t_{k_1\alpha+k_2\beta})=l_B(s_{\alpha}s_{\beta}s_{\alpha}t_{k_1\alpha+k_2\beta})+2k+1.$

We compute
$$N(t_{k_1\alpha+k_2\beta}s_{\alpha}s_{\beta}s_{\alpha}s_{\alpha+\beta+k\delta})$$
$$=\alpha_0^{-k-1+k_1-\frac{k_2}{2}}(-\alpha)_1^{k-k_1+\frac{k_2}{2}}\beta_0^{-k-1-\frac{k_1}{2}+k_2}(-\beta)_1^{k+\frac{k_1}{2}-k_2}$$
$$(\alpha+\beta)_0^{-2k-1+\frac{k_1+k_2}{2}}(-\alpha-\beta)_1^{2k-\frac{k_1+k_2}{2}}.$$
So $l_B(s_{\alpha+\beta+k\delta}s_{\alpha}s_{\beta}s_{\alpha}t_{k_1\alpha+k_2\beta})=k-k_1+\frac{k_2}{2}+k+\frac{k_1}{2}-k_2+2k-\frac{k_1+k_2}{2}=4k-k_1-k_2.$
Therefore $l_B(s_{\alpha+\beta+k\delta}s_{\alpha}s_{\beta}s_{\alpha}t_{k_1\alpha+k_2\beta})=l_B(s_{\alpha}s_{\beta}s_{\alpha}t_{k_1\alpha+k_2\beta})+4k+3.$
\end{proof}

We draw below part of the Hasse diagram of the poset $(\widetilde{W},\leq_B)$. The numeric sequence $a_1a_2\cdots a_t, a_i\in \{1,2,3\}$ stands for the reduced expression $s_{a_1}s_{a_2}\cdots s_{a_t}$ where $s_1:=s_{\alpha}, s_2:=s_{\beta}$ and $s_3:=s_{\delta-\alpha-\beta}.$ The black edge corresponds to a covering relation in the corresponding $B$-twisted left weak Bruhat order. The blue edge corresponds to the additional covering relation in $B$-twisted strong Bruhat order.

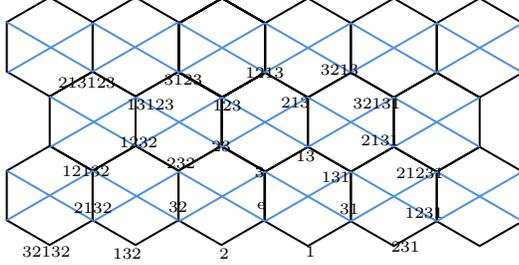
\tikzset{every picture/.style={line width=0.75pt}} 

\begin{tikzpicture}[x=0.75pt,y=0.75pt,yscale=-1,xscale=1]

\draw   (125,108) -- (146.65,120.5) -- (146.65,145.5) -- (125,158) -- (103.35,145.5) -- (103.35,120.5) -- cycle ;
\draw   (147,71) -- (168.65,83.5) -- (168.65,108.5) -- (147,121) -- (125.35,108.5) -- (125.35,83.5) -- cycle ;
\draw   (147,71) -- (168.65,83.5) -- (168.65,108.5) -- (147,121) -- (125.35,108.5) -- (125.35,83.5) -- cycle ;
\draw   (147,71) -- (168.65,83.5) -- (168.65,108.5) -- (147,121) -- (125.35,108.5) -- (125.35,83.5) -- cycle ;
\draw   (147,71) -- (168.65,83.5) -- (168.65,108.5) -- (147,121) -- (125.35,108.5) -- (125.35,83.5) -- cycle ;
\draw   (168.65,108.5) -- (190.3,121) -- (190.3,146) -- (168.65,158.5) -- (147,146) -- (147,121) -- cycle ;
\draw   (103.35,70.5) -- (125,83) -- (125,108) -- (103.35,120.5) -- (81.7,108) -- (81.7,83) -- cycle ;
\draw   (190.3,71) -- (211.95,83.5) -- (211.95,108.5) -- (190.3,121) -- (168.65,108.5) -- (168.65,83.5) -- cycle ;
\draw   (211.95,108.5) -- (233.6,121) -- (233.6,146) -- (211.95,158.5) -- (190.3,146) -- (190.3,121) -- cycle ;
\draw   (81.7,108) -- (103.35,120.5) -- (103.35,145.5) -- (81.7,158) -- (60.05,145.5) -- (60.05,120.5) -- cycle ;
\draw   (233.6,71) -- (255.25,83.5) -- (255.25,108.5) -- (233.6,121) -- (211.95,108.5) -- (211.95,83.5) -- cycle ;
\draw   (255.25,108.5) -- (276.9,121) -- (276.9,146) -- (255.25,158.5) -- (233.6,146) -- (233.6,121) -- cycle ;
\draw   (60.05,70.5) -- (81.7,83) -- (81.7,108) -- (60.05,120.5) -- (38.4,108) -- (38.4,83) -- cycle ;
\draw   (233.6,146) -- (255.25,158.5) -- (255.25,183.5) -- (233.6,196) -- (211.95,183.5) -- (211.95,158.5) -- cycle ;
\draw   (190.3,146) -- (211.95,158.5) -- (211.95,183.5) -- (190.3,196) -- (168.65,183.5) -- (168.65,158.5) -- cycle ;
\draw   (146.65,145.5) -- (168.3,158) -- (168.3,183) -- (146.65,195.5) -- (125,183) -- (125,158) -- cycle ;
\draw   (103.35,145.5) -- (125,158) -- (125,183) -- (103.35,195.5) -- (81.7,183) -- (81.7,158) -- cycle ;
\draw   (60.05,145.5) -- (81.7,158) -- (81.7,183) -- (60.05,195.5) -- (38.4,183) -- (38.4,158) -- cycle ;
\draw   (276.9,71) -- (298.55,83.5) -- (298.55,108.5) -- (276.9,121) -- (255.25,108.5) -- (255.25,83.5) -- cycle ;
\draw   (276.9,146) -- (298.55,158.5) -- (298.55,183.5) -- (276.9,196) -- (255.25,183.5) -- (255.25,158.5) -- cycle ;
\draw [color={rgb, 255:red, 74; green, 144; blue, 226 }  ,draw opacity=1 ]   (168.65,158.5) -- (211.95,183.5) ;

\draw [color={rgb, 255:red, 74; green, 144; blue, 226 }  ,draw opacity=1 ]   (125.35,158.5) -- (168.65,183.5) ;

\draw [color={rgb, 255:red, 74; green, 144; blue, 226 }  ,draw opacity=1 ]   (103.7,121) -- (147,146) ;

\draw [color={rgb, 255:red, 74; green, 144; blue, 226 }  ,draw opacity=1 ]   (81.7,158) -- (125,183) ;

\draw [color={rgb, 255:red, 74; green, 144; blue, 226 }  ,draw opacity=1 ]   (38.4,158) -- (81.7,183) ;

\draw [color={rgb, 255:red, 74; green, 144; blue, 226 }  ,draw opacity=1 ]   (211.95,158.5) -- (255.25,183.5) ;

\draw [color={rgb, 255:red, 74; green, 144; blue, 226 }  ,draw opacity=1 ]   (255.25,158.5) -- (298.55,183.5) ;

\draw [color={rgb, 255:red, 74; green, 144; blue, 226 }  ,draw opacity=1 ]   (60.05,120.5) -- (103.35,145.5) ;

\draw [color={rgb, 255:red, 74; green, 144; blue, 226 }  ,draw opacity=1 ]   (147,121) -- (190.3,146) ;

\draw [color={rgb, 255:red, 74; green, 144; blue, 226 }  ,draw opacity=1 ]   (190.3,121) -- (233.6,146) ;

\draw [color={rgb, 255:red, 74; green, 144; blue, 226 }  ,draw opacity=1 ]   (233.6,121) -- (276.9,146) ;

\draw [color={rgb, 255:red, 74; green, 144; blue, 226 }  ,draw opacity=1 ]   (211.95,83.5) -- (255.25,108.5) ;

\draw [color={rgb, 255:red, 74; green, 144; blue, 226 }  ,draw opacity=1 ]   (255.25,83.5) -- (298.55,108.5) ;

\draw [color={rgb, 255:red, 74; green, 144; blue, 226 }  ,draw opacity=1 ]   (168.65,83.5) -- (211.95,108.5) ;

\draw [color={rgb, 255:red, 74; green, 144; blue, 226 }  ,draw opacity=1 ]   (125.35,83.5) -- (168.65,108.5) ;

\draw [color={rgb, 255:red, 74; green, 144; blue, 226 }  ,draw opacity=1 ]   (82.05,83.5) -- (125.35,108.5) ;

\draw [color={rgb, 255:red, 74; green, 144; blue, 226 }  ,draw opacity=1 ]   (38.4,83) -- (81.7,108) ;

\draw [color={rgb, 255:red, 74; green, 144; blue, 226 }  ,draw opacity=1 ]   (298.55,158.5) -- (255.25,183.5) ;

\draw [color={rgb, 255:red, 74; green, 144; blue, 226 }  ,draw opacity=1 ]   (211.95,183.5) -- (255.25,158.5) ;

\draw [color={rgb, 255:red, 74; green, 144; blue, 226 }  ,draw opacity=1 ]   (168.65,183.5) -- (211.95,158.5) ;

\draw [color={rgb, 255:red, 74; green, 144; blue, 226 }  ,draw opacity=1 ]   (125.35,183.5) -- (168.65,158.5) ;

\draw [color={rgb, 255:red, 74; green, 144; blue, 226 }  ,draw opacity=1 ]   (82.05,183.5) -- (125.35,158.5) ;

\draw [color={rgb, 255:red, 74; green, 144; blue, 226 }  ,draw opacity=1 ]   (38.4,183) -- (81.7,158) ;

\draw [color={rgb, 255:red, 74; green, 144; blue, 226 }  ,draw opacity=1 ]   (233.6,146) -- (276.9,121) ;

\draw [color={rgb, 255:red, 74; green, 144; blue, 226 }  ,draw opacity=1 ]   (190.3,146) -- (233.6,121) ;

\draw [color={rgb, 255:red, 74; green, 144; blue, 226 }  ,draw opacity=1 ]   (147,146) -- (190.3,121) ;

\draw [color={rgb, 255:red, 74; green, 144; blue, 226 }  ,draw opacity=1 ]   (103.7,146) -- (147,121) ;

\draw [color={rgb, 255:red, 74; green, 144; blue, 226 }  ,draw opacity=1 ]   (60.4,146) -- (103.7,121) ;

\draw [color={rgb, 255:red, 74; green, 144; blue, 226 }  ,draw opacity=1 ]   (211.95,108.5) -- (255.25,83.5) ;

\draw [color={rgb, 255:red, 74; green, 144; blue, 226 }  ,draw opacity=1 ]   (168.65,108.5) -- (211.95,83.5) ;

\draw [color={rgb, 255:red, 74; green, 144; blue, 226 }  ,draw opacity=1 ]   (125.35,108.5) -- (168.65,83.5) ;

\draw [color={rgb, 255:red, 74; green, 144; blue, 226 }  ,draw opacity=1 ]   (81.7,108) -- (125,83) ;

\draw [color={rgb, 255:red, 74; green, 144; blue, 226 }  ,draw opacity=1 ]   (38.75,108.5) -- (82.05,83.5) ;

\draw [color={rgb, 255:red, 74; green, 144; blue, 226 }  ,draw opacity=1 ]   (255.25,108.5) -- (298.55,83.5) ;

\draw (82,176.5) node  [align=left] {{\tiny 2132}};
\draw (58.5,199) node  [align=left] {{\tiny 32132}};
\draw (99.2,199.8) node  [align=left] {{\tiny 132}};
\draw (124.8,176) node  [align=left] {{\tiny 32}};
\draw (148,199.8) node  [align=left] {{\tiny 2}};
\draw (166.8,175.8) node  [align=left] {{\tiny e}};
\draw (191.6,198.8) node  [align=left] {{\tiny 1}};
\draw (211.2,177.2) node  [align=left] {{\tiny 31}};
\draw (166,158.8) node  [align=left] {{\tiny 3}};
\draw (189.2,150.6) node  [align=left] {{\tiny 13}};
\draw (204.4,161) node  [align=left] {{\tiny 131}};
\draw (146.65,145.5) node  [align=left] {{\tiny 23}};
\draw (126.4,154) node  [align=left] {{\tiny 232}};
\draw (184,123.6) node  [align=left] {{\tiny 213}};
\draw (149.6,124.8) node  [align=left] {{\tiny 123}};
\draw (168.65,108.5) node  [align=left] {{\tiny 1213}};
\draw (226.8,142.6) node  [align=left] {{\tiny 2131}};
\draw (225.2,124.2) node  [align=left] {{\tiny 32131}};
\draw (206.4,107.2) node  [align=left] {{\tiny 3213}};
\draw (105,143.5) node  [align=left] {{\tiny 1232}};
\draw (78.5,158) node  [align=left] {{\tiny 12132}};
\draw (127.5,112) node  [align=left] {{\tiny 3123}};
\draw (111,124.5) node  [align=left] {{\tiny 13123}};
\draw (239.5,196.5) node  [align=left] {{\tiny 231}};
\draw (79,113.5) node  [align=left] {{\tiny 213123}};
\draw (249,179) node  [align=left] {{\tiny 1231}};
\draw (247,159) node  [align=left] {{\tiny 21231}};

\end{tikzpicture}

We collect in the following proposition some properties and the symmetry of this poset.

\begin{proposition}\label{posetprop}
(1) The sphericity of the intervals is characterized  as follow.

(a) The length 2 intervals of $(\widetilde{W},\leq_B)$ are either of the form $[x,s_1s_2x]$ where $s_1\neq s_2, s_1,s_2\in \{s_{\alpha},s_{\beta},s_{\delta-\alpha-\beta}\}$ or of the form $[x,s_{\delta-\gamma}s_{\gamma}x]$ where $\gamma\in \{\alpha,\beta,\alpha+\beta, \delta-\alpha, \delta-\beta, \delta-\alpha-\beta\}$. The former is spherical and the latter is non-spherical.

(b) The length 3 intervals of $(\widetilde{W},\leq_B)$ are all non-spherical except those of the following types:
$[x,s_{\alpha}s_{\beta}s_{\alpha}x]$ where $x\in s_{\alpha}s_{\beta}s_{\alpha}T$; $[x, s_{\delta-\alpha-\beta}s_{\beta}s_{\delta-\alpha-\beta}x]$ where $x\in s_{\beta}T$; $[x, s_{\alpha}s_{\delta-\alpha-\beta}s_{\alpha}x]$ where $x\in s_{\alpha}T$.

(c) All intervals of $(\widetilde{W},\leq_B)$ of length greater than 3 are non-spherical.

(2) Let $x=(s_{\beta}s_{\delta-\alpha-\beta}s_{\alpha})^{\infty}$ and $y=(s_{\alpha}s_{\delta-\alpha-\beta}s_{\beta})^{\infty}$ be two infinite reduced words. For $i\geq 0$, denote $x(i)$ and $x(-i)$ the unique length $i$ prefixes of $x$ and $y$ respectively.

Let $U$ be the (infinite dihedral) reflection subgroup of $\widetilde{W}$ generated by $v:=s_{\delta-\alpha-\beta}$ and $u:=s_{\alpha}s_{\beta}s_{\alpha}$. For $x\in \widetilde{W}$,  $y$ is the unique minimal element of the left coset $xU$ if and only if $y=w(i)^{-1}$ for some  $i\in \mathbb{Z}$. Consequently every element $w\in \widetilde{W}$ can be written in the form $w(i)^{-1}u$ for some $u\in U$ and $i\in \mathbb{Z}$.

Furthermore for every $w\in \widetilde{W}$, $l_B(w)$ is determined as follow:

(a) $l_B(w(i)^{-1}u(vu)^k)=-4k-3$ if $i$ is even and $l_B(w(i)^{-1}u(vu)^k)=-4k-2$ if $i$ is odd.

(b) $l_B(w(i)^{-1}(vu)^k)=-4k$ if $i$ is even and $l_B(w(i)^{-1}(vu)^k)=-4k-1$ if $i$ is odd.

(c) $l_B(w(i)^{-1}v(uv)^k)=4k+1$ if $i$ is even and $l_B(w(i)^{-1}v(uv)^k)=4k+2$ if $i$ is odd.

(d) $l_B(w(i)^{-1}(uv)^k)=4k$ if $i$ is even and $l_B(w(i)^{-1}(uv)^k)=4k-1$ if $i$ is odd.

(3) Suppose that $x,y\in \widetilde{W}$ such that $l_B(x)$ and $l_B(y)$ are of the same parity. Then there exists a poset automorphism $\gamma: (\widetilde{W},\leq_B)\rightarrow (\widetilde{W},\leq_B)$ such that $\gamma(x)=y.$
\end{proposition}

\begin{proof}
(1) This can be verified routinely by using the covering relations in Lemma \ref{sixtype}.

(2) By \cite{DyerReflSubgrp} $y$ is an element that satisfies $N(y^{-1})\cap\Phi_{U}^+=\emptyset.$ Hence $N(y^{-1})\subset \{\alpha,-\beta\}^{\wedge}\uplus \{-\alpha,\beta\}^{\wedge}$. Note that $\{\alpha,-\beta\}^{\wedge}=(s_{\beta}s_{\delta-\alpha-\beta}s_{\alpha})^{\infty}$ and $ \{-\alpha,\beta\}^{\wedge}=(s_{\alpha}s_{\delta-\alpha-\beta}s_{\beta})^{\infty}$.

Assume that $N(y^{-1})\not\subset N(s_{\beta}s_{\delta-\alpha-\beta}s_{\alpha})^{\infty})$ and $N(y^{-1})\not\subset N((s_{\alpha}s_{\delta-\alpha-\beta}s_{\beta})^{\infty})$. Without loss of generality assume that $\alpha\in N(y^{-1})$.
Then $\{-\alpha\}^{\wedge}\cap N(y^{-1})=\emptyset$. This forces some $\beta+k\delta\in N(y^{-1})$. Then $\alpha+\beta+k\delta\in N(y^{-1})$, contradicting $N(y^{-1})\cap\Phi_{U}^+=\emptyset.$ Hence either $N(y^{-1})\subset N(s_{\beta}s_{\delta-\alpha-\beta}s_{\alpha})^{\infty})$ or $N(y^{-1})\subset N((s_{\alpha}s_{\delta-\alpha-\beta}s_{\beta})^{\infty})$.

Hence given $w\in \widetilde{W}$, then $w(i)^{-1}=wu$ for some $i\in \mathbb{Z}$ and $u\in U.$

Now we check the results on twisted length. We prove (a) and (b)-(d) can be verified similarly. One computes that

$N((uv)^kuw(i))=\alpha_0^{k-\lceil\frac{i}{2}\rceil}\beta_0^{k+\lfloor\frac{i}{2}\rfloor}(\alpha+\beta)_0^{2k}(-\alpha)_1^{\lceil\frac{i}{2}\rceil-k-1}(-\beta)_1^{\lceil-\frac{i}{2}\rceil-k-1}.$

There are six possiblities depending on $k$ and $i$.

(i) $N((uv)^kuw(i))=\alpha_0^{k-\lceil\frac{i}{2}\rceil}\beta_0^{k+\lfloor\frac{i}{2}\rfloor}(\alpha+\beta)_0^{2k}.$ Then $l_B((uv)^kuw(i))=-4k-3+\lceil\frac{i}{2}\rceil-\lfloor\frac{i}{2}\rfloor$. If $i$ is even the twisted length is $-4k-3$. If $i$ is odd, the twisted length is $-4k-2.$

(ii) $N((uv)^kuw(i))=\alpha_0^{k-\lceil\frac{i}{2}\rceil}(\alpha+\beta)_0^{2k}(-\beta)_1^{\lceil-\frac{i}{2}\rceil-k-1}.$ Then $l_B((uv)^kuw(i))=-4k-3+\lceil\frac{i}{2}\rceil+\lceil-\frac{i}{2}\rceil$. If $i$ is even the twisted length is $-4k-3$. If $i$ is odd, the twisted length is $-4k-2.$

(iii) $N((uv)^kuw(i))=\beta_0^{k+\lfloor\frac{i}{2}\rfloor}(\alpha+\beta)_0^{2k}(-\alpha)_1^{\lceil\frac{i}{2}\rceil-k-1}.$ Then $l_B((uv)^kuw(i))=-4k-3+\lceil\frac{i}{2}\rceil-\lfloor\frac{i}{2}\rfloor$. If $i$ is even the twisted length is $-4k-3$. If $i$ is odd, the twisted length is $-4k-2.$

(iv) $N((uv)^kuw(i))=(\alpha+\beta)_0^{2k}(-\alpha)_1^{\lceil\frac{i}{2}\rceil-k-1}(-\beta)_1^{\lceil-\frac{i}{2}\rceil-k-1}.$
Then $l_B((uv)^kuw(i))=-4k-3+\lceil\frac{i}{2}\rceil+\lceil-\frac{i}{2}\rceil$. If $i$ is even the twisted length is $-4k-3$. If $i$ is odd, the twisted length is $-4k-2.$

(v) If $k-\lceil\frac{i}{2}\rceil=-1$, $N((uv)^kuw(i))=\beta_0^{k+\lfloor\frac{i}{2}\rfloor}(\alpha+\beta)_0^{2k}.$ Then $l_B((uv)^kuw(i))=-3k-2-\lfloor\frac{i}{2}\rfloor=-4k-3+\lceil\frac{i}{2}\rceil-\lfloor\frac{i}{2}\rfloor$. If $i$ is even the twisted length is $-4k-3$. If $i$ is odd, the twisted length is $-4k-2.$

(vi) If $k+\lfloor\frac{i}{2}\rfloor=-1$, $N((uv)^kuw(i))=\alpha_0^{k-\lceil\frac{i}{2}\rceil}(\alpha+\beta)_0^{2k}.$ Then $l_B((uv)^kuw(i))=-3k-2+\lceil\frac{i}{2}\rceil=-4k-3+\lceil\frac{i}{2}\rceil-\lfloor\frac{i}{2}\rfloor$. If $i$ is even the twisted length is $-4k-3$. If $i$ is odd, the twisted length is $-4k-2.$

(3) Let $\sigma$  be the group automorphism of $\widetilde{W}$ induced by the map $s_{\delta-\alpha-\beta}\mapsto s_{\alpha}, s_{\beta}\mapsto s_{\delta-\alpha-\beta}, s_{\alpha}\mapsto s_{\beta}$.

Note that $\sigma(t_{\alpha^{\vee}})=\sigma(s_{\alpha}s_{\beta}s_{\delta-\alpha-\beta}s_{\beta})=s_{\beta}s_{\delta-\alpha-\beta}s_{\alpha}s_{\delta-\alpha-\beta}=t_{\beta^{\vee}}$ and $\sigma(t_{\beta^{\vee}})=\sigma(s_{\beta}s_{\alpha}s_{\delta-\alpha-\beta}s_{\alpha})=s_{\delta-\alpha-\beta}s_{\beta}s_{\alpha}s_{\beta}=t_{-(\alpha+\beta)^{\vee}}$. So we conclude $\sigma(T)=T$.

We show that there exists a poset automorphism $\eta$ of $(\widetilde{W},\leq_B)$ given by $w\mapsto \sigma(w)s_{\alpha}s_{\beta}$ and a poset automorphism $\eta'$ of $(\widetilde{W},\leq_B)$ given by $w\mapsto \sigma^{-1}(w)s_{\beta}s_{\alpha}$.

 Using the fact that $T$ is a normal subgroup one obtains that
$$\eta(w)\in s_{\alpha}s_{\beta}T,$$
$$\eta(s_{\alpha}s_{\beta}w)\in s_{\beta}s_{\alpha}T,$$
$$\eta(s_{\beta}s_{\alpha}w)\in T,$$
$$\eta(s_{\beta}s_{\alpha}s_{\beta}w)\in s_{\beta}T,$$
$$\eta(s_{\beta}T)\in s_{\alpha}T,$$
$$\eta(s_{\alpha}T)\in s_{\beta}s_{\alpha}s_{\beta}T.$$

By above and the covering relations given in Lemma \ref{sixtype} one can easily checks that $\eta$ preserves covering relations. Similarly one verifies that $\eta'$ also preserves covering relations. Let $\rho=\eta(\eta')^{-1}$ be the poset automorphism of $(\widetilde{W},\leq_B)$.

For $k\in \mathbb{Z},$ $\rho^k(w)=\sigma^{2k}(w)(x(2k))^{-1}$.
For any $z\in U$ (as in (2)) we have
$$\rho^2(w(i)^{-1}z)=\sigma^{-1}(w(i)^{-1}z)x(2)^{-1}=w(i+2)^{-1}(z).$$

By (2) if $l_B(x)=l_B(y)$, $x=w(j)^{-1}z$ and $y=w(k)^{-1}z$ where $j$ and $k$ are of the same parity and $z\in U$. Then for some $t$, $\rho^t(x)=y$.

Now assume $l_B(x)\neq l_B(y)$ but they are of the same parity. Let $v,u$ be as in (2).
We compute
$$\eta(w(i)^{-1}u(vu)^k)=w(i+1)^{-1}(vu)^{k+1},$$
$$\eta(w(i)^{-1}(vu)^k)=w(i+1)^{-1}u(vu)^k,$$
$$\eta(w(i)^{-1}v(uv)^k)=w(i+1)^{-1}(uv)^k,$$
$$\eta(w(i)^{-1}(uv)^k)=w(i+1)^{-1}v(uv)^{k-1}.$$

Hence for any $z\in \widetilde{W}.$ $l_B(\eta(z))=l_B(z)-2.$
Therefore for some $t,$ $l_B(x)=l_B(\eta^t(y))$.
\end{proof}

\begin{remark}
One of the classical theme in the study of Coxeter group  (or more generally combinatorics) is the enumeration problem.
Given the local finiteness of the twisted Bruhat order on affine Weyl groups, one can
define the Poincar\'e series associated to the twisted Bruhat order $\leq_B$ and an element $w$ to be $F_{w,B}(t)=\sum_{u\leq_B w}t^{l(w)-l(u)}$. This generalizes the usual Poincar\'e polynomial associated to the ordinary Bruhat order.
For type $\widetilde{A}_2$ and the above biclosed set $B$, we have the following result.

If $l(w)$ is even, $F_{w,B}(t)=\frac{t^3+2t^2+2t+1}{t^4-2t^2+1}$. If $l(w)$ is odd, $F_{w,B}(t)=\frac{2t^2+3t+1}{t^4-2t^2+1}$.

To see this, one can apply the Principle of Inculsion-Exclusion and Proposition \ref{posetprop} (3) to obtain
$$F_1(t)=1+2tF_2(t)-2t^2F_1(t)+t^3F_2(t)$$
$$F_2(t)=1+3tF_1(t)-2t^2F_2(t)$$
where $F_1(t)=F_{w,B}(t)$ for $w$ even and $F_2(t)=F_{w,B}(t)$ for $w$ odd.
Solve for $F_1(t)$ and $F_2(t)$ one obtains  the closed formlae.
\end{remark}

\section{Tope Poset of the affine root system and a Bj\"{o}rner-Edelman-Ziegler type theorem}\label{omsec}

The root system of a Coxeter group can be naturally regarded as an oriented matroid. The affine root systems provide interesting and accessible examples of infinite oriented matroids in the sense of \cite{largeconvex}.
In this section we will use twisted weak order as a tool to investigate the tope poset of the oriented matroid arising from an affine root system.

 An (possibly infinite) oriented matroid is a triple $(E,*,cx)$  where $E$ is a set with an involution map $*: E\rightarrow E$ (i.e. $x^{**}=x, x\neq x^*$) and $cx$ a closure operator on $E$ such that

(i) if $x\in cx(X)$ there exists a finite set $Y\subseteq  X$ such that $x\in cx(Y)$,

(ii) $cx(X)^*=cx(X^*)$,

(iii) if $x\in cx(X\cup \{x^*\})$ then $x\in cx(X)$,

(iv) if $x\in cx(X\cup \{y^*\})$ and $x\not\in cx(X)$ then $y\in cx(X\backslash\{y\}\cup\{x^*\})$.

Given a root system $\Phi$, $(\Phi,-,\text{cone}_{\Phi})$ is an oriented matroid where $\text{cone}_{\Phi}=\text{cone}(\Gamma)\cap \Phi$ and $\cone(\Gamma)=\{\sum_{i\in I}k_iv_i|v_i\in \Gamma\cup\{0\}, k_i\in \mathbb{R}_{\geq 0}, |I|<\infty\}$.

A $cx$-closed set  in $E$ is said to be convex.
A convex set $H\subset E$ is called a hemispace or a tope if $H\cap H^*=\emptyset, H\cup H^*=E$.
Suppose that $\Phi$ is the root system of a finite Coxeter system. Then every positive system is a hemispace of the oriented matroid $(\Phi,-,\text{cone}_{\Phi})$ and vice versa. We denote the set of hemispaces by $\mathcal{H}(E).$
For a finite oriented matroid $(E,*,cx)$
and a hemispace $H$, there exists a unique minimal subset of $H$ whose $cx-$closure is $H$. We denote such a set by $\mathrm{ex}(H)$ and call it the set of extreme elements of $H.$
A hemispace is said to be simplicial if $|\mathrm{ex}(H)|$ equals the rank of the underlying (unoriented) matroid of $(E,*,cx)$. A finite oriented matroid is said to be simplicial if all of its hemispaces are simplicial.

Denote by $\Delta$ the set symmetric difference. Fix a hemispace $H$ of $(E,*,cx)$ and one can define a poset structure on the set of hemispaces by $F\leq G\Longleftrightarrow F\Delta H \subset G\Delta H, F,G\in \mathcal{H}(E).$ We call such a poset the tope poset based at $H.$ In a tope poset $P$ of an infinite oriented matroid, we say two hemisapces $H_1, H_2$ are in the same block of $P$ if their symmetric difference is finite.

In \cite{hyperplane}, Bj\"{o}rner, Edelman and Ziegler proved the following

\begin{theorem}
For a finite simplicial oriented matroid, the tope poset for any choice of base hemispace is a lattice.
\end{theorem}

It is well known that for the root system of a finite Coxeter system, the corresponding oriented matroid is simplicial. In fact all tope posets are isomorphic to the ordinary right weak  order of the Coxeter group (which is a lattice). In this section we prove the following analogous theorem for irreducible affine root systems.

\begin{theorem}\label{intervallattice}
Let $\Phi$ be a finite irreducible crystallographic root system and let $H$ be any hemispace of the oriented matroid $(\widetilde{\Phi},-,cx)$  where $cx:=\mathrm{cone}_{\widetilde{\Phi}}$. Let $H_1, H_2$ be two hemispaces in a same block and $H_1\leq H_2$ in the tope poset based at $H$. Then $[H_1, H_2]$ is a lattice.
\end{theorem}

We will first classify all hemispaces of $(\widetilde{\Phi},-,cx)$.

Let $\Gamma$ be a subset of a real vector space $V$. A subset $A$ of $\Gamma$ is called biconvex in $\Gamma$ if $A$ and $\Gamma\backslash A$ are both convex in the oriented matroid $(\Gamma,-,\mathrm{cone}_{\Gamma})$.

For the next lemma, let $\Phi$ be a finite irreducible crystallographic root system and let $\Psi^+$ be a positive system of $\Phi$.

\begin{lemma}\label{lem:coneintersectzero}
Let $P(\Psi^+,\Delta_1,\Delta_2)$ be a 2 closure biclosed set in $\Phi$. Then $\mathbb{R}\Delta_1\cap \mathrm{cone}(P(\Psi^+,\Delta_1,\Delta_2))=0$ and the 2 closure biclosed set $P(\Psi^+,\Delta_1,\Delta_2)$ is biconvex in $\Phi$.
\end{lemma}

\begin{proof}
Take a linear function $f:\mathbb{R}\Phi\rightarrow \mathbb{R}$ such that $f$ is positive on $\Delta\backslash (\Delta_1\cup\Delta_2)$ and zero on $\Delta_1\cup \Delta_2.$ Then $f$ is non-negative on $P(\Psi^+,\Delta_1,\Delta_2)$.
If $\gamma_1,\cdots,\gamma_k\in P(\Psi^+,\Delta_1,\Delta_2), c_1,\cdots,c_k\in \mathbb{R}_{>0}$ and $\gamma=\sum_{i=1}^kc_i\gamma_i\in \mathbb{R}\Delta_1,$ then apply $f$ on both sides we see all $\gamma_i$ are in $\mathbb{R}(\Delta_1\cup \Delta_2)$ as $f(\gamma)=0.$ Then they are in either $\mathbb{R}\Delta_1$ or $\mathbb{R}\Delta_2$ as $\Delta_1\perp\Delta_2.$ But since they are in $P(\Psi^+,\Delta_1,\Delta_2)$, they have to be in $\mathbb{R}\Delta_2.$ Then this forces $\gamma=0.$

To prove the second assertion, again let $\gamma_1,\cdots,\gamma_k\in P(\Psi^+,\Delta_1,\Delta_2), c_1,\cdots,c_k\in \mathbb{R}_{>0}$ and $\gamma=\sum_{i=1}^kc_i\gamma_i$. Note that $\Phi\backslash P(\Psi^+,\Delta_1,\Delta_2)=\mathbb{R}\Delta_1\uplus -P(\Psi^+,\Delta_1\cup\Delta_2,\emptyset).$ Suppose $\gamma\in -P(\Psi^+,\Delta_1\cup\Delta_2,\emptyset)$, apply $f$ on $\gamma$ one gets a negative number, which is a contradiction. Therefore combining this with the first assertion,  one sees $P(\Psi^+,\Delta_1,\Delta_2)$ is convex.
Note $\Phi\backslash P(\Psi^+,\Delta_1,\Delta_2)=P(\Psi^-,-\Delta_1,-\Delta_2)$ and therefore is convex as well. Hence we see that $P(\Psi^+,\Delta_1,\Delta_2)$ is biconvex in $\Phi$.
\end{proof}

Similarly one can define the notion of a 2 closure hemispace in a root system $\Phi$. A 2 closure hemispace in $\Phi$ is a 2 closure closed set $H$ such that $H\uplus -H=\Phi$. A hemispace of the oriented matroid $(\Phi,-,cx)$ is necessarily a 2 closure hemispace of $\Phi$ but not vice versa.
There exists a bijection between the set of 2 closure biclosed sets in $\Phi^+$ and the set of 2 closure hemispaces under the map $B\mapsto B\uplus -(\Phi^+\backslash B)$ for $B$ a 2 closure biclosed set in $\Phi^+.$ Since the 2 closure biclosed sets in $\widetilde{\Phi}^+$ are classified, the corresponding 2 closure hemispaces are therefore known via the bijection. We will next examine which of these 2 closure hemispaces are indeed hemispaces of the oriented matroid $(\widetilde{\Phi},-,cx)$.

\begin{lemma}\label{lem:finitehemispace}
For any Coxeter system $(W,S)$ and $w\in W$,

(1) the 2-closure hemispace $N(w)\uplus -(\Phi^+\backslash N(w))$ is convex,

(2) the 2-closure hemispace $N(w)'\uplus -(\Phi^+\backslash N(w)')$ is convex.
\end{lemma}

\begin{proof}
(1) follows from the fact that $N(w)\uplus -(\Phi^+\backslash N(w))=-w\Phi^+$ and that $\Phi^+$ is convex.
(2) follows from the fact that $N(w)'\uplus -(\Phi^+\backslash N(w)')=w\Phi^+$ and that $\Phi^+$ is convex.
\end{proof}

Now let $P(\Psi^+,\Delta_1,\Delta_2)$ be a 2 closure biclosed set (also biconvex by Lemma \ref{lem:coneintersectzero}) in $\Phi$. Any 2 closure biclosed set in $(\Phi)^{\wedge}$ is of the form $(P(\Psi^+,\Delta_1,\Delta_2)^{\wedge}\backslash B_2)\cup B_1$
where $B_1$ (resp. $B_2$) be a finite 2 closure biclosed set in $(\Phi_1)^{\wedge}$ (resp. $(\Phi_2)^{\wedge}$) and $\Phi_i=\mathbb{R}\Delta_i\cap \Phi$.
Then we can associate to it a 2 closure hemispace  $H:=B\uplus -((\Phi)^{\wedge}\backslash B)$. We  investigate when this 2 closure hemispace is a cone closure (i.e. oriented matroidal) hemispace.

\begin{proposition}\label{hemispacecoincide}
(1) $H$ is convex if $\Delta_2$ is empty or $\Delta_1$ is empty.

(2) If both $\Delta_1$ and $\Delta_2$ are non-empty, $H$ is not convex.
\end{proposition}

\begin{proof}
(1) Assume $\Delta_2=\emptyset.$
Suppose $\{\pm\alpha_1,\pm\alpha_2,\cdots,\pm\alpha_m\}=\mathbb{R}\Delta_1\cap \Phi=\Phi_1.$
Let $\{\gamma_1,\gamma_2,\cdots,\gamma_p\}=\Psi^+\backslash \Phi_1$.

By the construction $H$ contains $(\gamma_1)_{-\infty}^{\infty},(\gamma_2)_{-\infty}^{\infty},\cdots,(\gamma_p)_{-\infty}^{\infty}$.

Suppose $\alpha_{i_1}+k_1\delta, \alpha_{i_2}+k_2\delta, \cdots, \alpha_{i_t}+k_t\delta, \gamma_{j_1}+l_1\delta, \gamma_{j_2}+l_2\delta,\cdots, \gamma_{j_s}+l_s\delta\in H$ where
$\alpha_{i_p}\in \Phi_1, 1\leq p\leq t.$ Take $a_1, a_2, \cdots, a_t, b_1, b_2, \cdots, b_s\in \mathbb{R}_{>0}$. Assume that
\begin{equation}\label{imp}
\sum_{u=1}^t a_u(\alpha_{i_u}+k_u\delta)+\sum_{v=1}^s b_v(\gamma_{j_v}+l_v\delta)=\alpha'+q\delta
\end{equation}
where $\alpha'\in \Phi_1$ and $\alpha'+q\delta\not\in H.$ Then one has
$$\alpha'-\sum_{u=1}^t a_u\alpha_{i_u}=\sum_{v=1}^s b_v\gamma_{j_v}.$$
By Lemma \ref{lem:coneintersectzero} both sides have to be $0$. This forces all $b_v=0$ (as all $\gamma_i$ are in a positive system). So we have
$$\sum_{u=1}^t a_u(\alpha_{i_u}+k_u\delta)=\alpha'+q\delta.$$
On the other hand, note that $H\cap \widetilde{\Phi_1}$ is a  hemispace in $(\widetilde{\Phi_1},-,\mathrm{cone}_{\widetilde{\Phi_1}})$. Indeed it is $B_1\cup -((\Phi_1)^{\wedge}\backslash B_1)$.
Since $B_1$ is finite and biclosed in $(\Phi_1)^{\wedge}$, it is of the form $\Phi_w$ where $w$ is an element of the reflection subgroup generated by the reflections corresponding to the roots in $(\Phi_1)^{\wedge}$. So Lemma \ref{lem:finitehemispace} (1) ensures the above equation \ref{imp} is not possible.

Now assume that
\begin{equation}\label{imp2}
\sum_{u=1}^t a_u(\alpha_{i_u}+k_u\delta)+\sum_{v=1}^s b_v(\gamma_{j_v}+l_v\delta)=-\gamma_w+q\delta.
\end{equation}

Then one has
$$-\sum_{u=1}^t a_u\alpha_{i_u}=\sum_{v=1}^s b_v\gamma_{j_v}+\gamma_w.$$

Again by Lemma \ref{lem:coneintersectzero} both sides have to be $0$. But all $\gamma_v$ are in a positive system of $\Phi.$ So the equation \ref{imp2} is not possible.

Hence $H$ is indeed convex.

Now assume $\Delta_1=\emptyset$. Then same type of reasoning as above (using Lemma \ref{lem:finitehemispace} (2)) shows $H$ is convex.

(2) Now we assume that both $\Delta_1$ and $\Delta_2$ are non-empty.

Again suppose that $\{\pm\alpha_1,\pm\alpha_2,\cdots,\pm\alpha_m\}=\mathbb{R}\Delta_1\cap \Phi=\Phi_1$ and assume that $\alpha_1,\alpha_2,\cdots, \alpha_m\in \Phi^+.$
Suppose that $\{\pm\beta_1, \pm\beta_2, \cdots, \pm\beta_l\}=\mathbb{R}\Delta_2\cap \Phi=\Phi_2$ and assume that $\beta_1, \beta_2, \cdots, \beta_l\in \Phi^+.$

By Lemma 4.1 of \cite{wang}, $(\{\beta_i+t\delta|t\in \mathbb{Z}\}\cup \{-\beta_i+t\delta|t\in \mathbb{Z}\})\cap (P(\Psi^+,\Delta_1,\Delta_2)^{\wedge}\backslash B_2)\cup B_1$
has four possibilities: 

(i) $(\beta_i)_p^{\infty}\cup (-\beta_i)_1^{\infty}, p\geq 0,$

(ii) $(\beta_i)_p^{\infty}\cup (-\beta_i)_0^{\infty}, p\geq 1,$

(iii) $(-\beta_i)_p^{\infty}\cup (\beta_i)_1^{\infty}, p\geq 0,$

(iv) $(-\beta_i)_p^{\infty}\cup (\beta_i)_0^{\infty}, p\geq 1.$

Correspondingly $(\{\beta_i+t\delta|t\in \mathbb{Z}\}\cup \{-\beta_i+t\delta|t\in \mathbb{Z}\})\cap H$ can be one of the following four possibilities:

(i) $(\beta_i)_p^{\infty}\cup (-\beta_i)_{1-p}^{\infty}, p\geq 0,$

(ii) $(\beta_i)_p^{\infty}\cup (-\beta_i)_{1-p}^{\infty}, p\geq 1,$

(iii) $(-\beta_i)_p^{\infty}\cup (\beta_i)_{1-p}^{\infty}, p\geq 0,$

(iv) $(-\beta_i)_p^{\infty}\cup (\beta_i)_{1-p}^{\infty}, p\geq 1.$

We note in each of these four cases for sufficiently large $s,t\in \mathbb{Z}_{>0}, \beta_i+s\delta$ and $-\beta_i+t\delta\in H$. It follows that $\delta\in \cone(H)$.

Similarly by Lemma 4.1 of \cite{wang}, $(\{\alpha_i+t\delta|t\in \mathbb{Z}\}\cup \{-\alpha_i+t\delta|t\in \mathbb{Z}\})\cap\newline (P(\Psi^+,\Delta_1,\Delta_2)^{\wedge}\backslash B_2)\cup B_1$
has four possibilities:

(i) $(\alpha_i)_0^p, p\geq 0$,

(ii) $(-\alpha_i)_0^p, p\geq 0,$

(iii) $(\alpha_i)_1^p, p\geq 1,$

(iv) $(-\alpha_i)_1^p, p\geq 1.$

Correspondingly $(\{\alpha_i+t\delta|t\in \mathbb{Z}\}\cup \{-\alpha_i+t\delta|t\in \mathbb{Z}\})\cap H$ can be one of the following four possibilities:

(i) $(\alpha_i)_{-\infty}^p\cup (-\alpha_i)_{-\infty}^{-p-1}, p\geq 0,$

(ii) $(-\alpha_i)_{-\infty}^p\cup (\alpha_i)_{-\infty}^{-p-1}, p\geq 0,$

(iii) $(\alpha_i)_{-\infty}^p\cup (-\alpha_i)_{-\infty}^{-p-1}, p\geq 1,$

(iv) $(-\alpha_i)_{-\infty}^p\cup (\alpha_i)_{-\infty}^{-p-1}, p\geq 1.$

We note in each of these four cases  it follows that there exist $s,t\in \mathbb{Z}_{<0}, \beta_i+s\delta$ and $-\beta_i+t\delta\in H$. Hence $-\delta\in \cone(H)$.

Thus we see that $(\{\beta_i+t\delta|t\in \mathbb{Z}\}\cup\{-\beta_i+t\delta|t\in \mathbb{Z}\})$ (resp.  $(\{\alpha_i+t\delta|t\in \mathbb{Z}\}\cup\{-\alpha_i+t\delta|t\in \mathbb{Z}\})$) is completely contained in $\cone(H).$ Therefore $H$ is not convex.
\end{proof}

Now by Theorem \ref{infiniteword} we immediately obtain

\begin{corollary}\label{twotypes}
For an irreducible affine Weyl group $\widetilde{W}$, any hemispace of the oriented matroid $(\widetilde{\Phi},-,\mathrm{cone}_{\widetilde{\Phi}})$ is of the form $N(w)\uplus -(\Phi^{\wedge}\backslash N(w))$ or $-N(w)\uplus (\Phi^{\wedge}\backslash N(w))$ where $w\in \widetilde{W}$ or is an infinite reduced word.
\end{corollary}

\begin{remark}
Let $(W,S)$ be a rank 3 universal Coxeter system. The three simple roots are denoted by $\alpha_1,\alpha_2,\alpha_3.$ Then $B:=\{\alpha\in \Phi^+|\alpha=k_1\alpha_1+k_2\alpha_2\}$ is a biclosed set in $\Phi^+$. One easily sees that $H=B\uplus -(\Phi^+\backslash B)=\{\alpha\in \Phi^+|\alpha=k_1\alpha_1+k_2\alpha_2\}\uplus \{\alpha\in \Phi^-|\alpha=k_1\alpha_1+k_2\alpha_2+k_3\alpha_3, k_3<0\}$ is a matroidal hemispace of the oriented matroid $(\Phi,-,cone_{\Phi})$. However in this case neither $B$ nor $\Phi^+\backslash B$ is an inversion set.
\end{remark}

\begin{lemma}\label{torder}
Let $B=P(\Psi^+,\Delta_1,\emptyset)^{\wedge}$ be a biclosed set in $(\Phi)^{\wedge}$.
Let $H=B\uplus -((\Phi)^{\wedge}\backslash B)$ be a hemispace of the oriented matroid  $(\widetilde{\Phi},-,cx)$ and $G$ is an arbitrary hemispace of $(\widetilde{\Phi},-,cx)$.
Then the block containing $H$ of the tope poset based at $G$ is isomorphic to the twisted weak order $(W',G\cap \widetilde{\Phi_1})$ where $W'$ is the reflection subgroup generated by $\widetilde{\Phi_1}$ and $\widetilde{\Phi_1}=\pm(\mathbb{R}\Delta_1)^{\wedge}.$
\end{lemma}

\begin{proof}
First note that since $G$ is a hemispace in the oriented matroid, it is biclosed in $\widetilde{\Phi}$. Therefore $G\cap \widetilde{\Phi_1}$ is indeed a biclosed set in $\widetilde{\Phi_1}.$
Let $H_1=B_1\uplus -(\Phi^{\wedge}\backslash B_1)$ be another hemispace. If $H$ and $H_1$ are in the same block, then clearly the symmetric difference of $B$ and $B_1$ is finite.

By Proposition 5.11 and Corollary 5.12 of \cite{DyerReflOrder} (See also Theorem 1.3 of \cite{wang}), $B_1=w\cdot B$ for some unique $w\in W'$. Then $H_1=w\cdot B\uplus -(\Phi^{\wedge}\backslash w\cdot B)=wH$ by \cite{BPK} 3.2(e).
Now define the map from the block containing $H$ of the tope poset based at $G$ to $(W',G\cap \widetilde{\Phi_1})$ by $wH\mapsto w$.

We must show that such a map and its inverse preserve the orders.
Take $w_1H, w_2H$ such that $w_1H<w_2H$ in the block containing $H$ of the tope poset based at $G$.
Note that $w_iH=w_i\cdot P(\Psi^+,\Delta_1,\emptyset)^{\wedge} \uplus -((\Phi)^{\wedge}\backslash w_i\cdot P(\Psi^+,\Delta_1,\emptyset)^{\wedge}), i=1,2.$ by \cite{BPK} 3.2(e).
Further one has that $w_i\cdot P(\Psi^+,\Delta_1,\emptyset)^{\wedge}=P(\Psi^+,\Delta_1,\emptyset)^{\wedge}\uplus (N(w_i)\cap \widetilde{\Phi_1})$ by Proposition 5.11  of \cite{DyerReflOrder}.
Hence $w_iH=H\uplus (N(w_i)\cap \widetilde{\Phi_1})\backslash -((N(w_i)\cap \widetilde{\Phi_1}))$. Then $w_1H\Delta G\subset w_2H\Delta G$ if and if the following two containments hold:
$$(N(w_1)\cap \widetilde{\Phi_1})\backslash G\subset (N(w_2)\cap \widetilde{\Phi_1})\backslash G,$$
$$(N(w_2)\cap \widetilde{\Phi_1})\cap G\subset (N(w_1)\cap \widetilde{\Phi_1})\cap G.$$
One sees that these two containments are equivalent to the conditions $(N(w_1)\cap \widetilde{\Phi_1})\backslash (N(w_2)\cap \widetilde{\Phi_1})\subset G$ and $(N(w_2)\cap \widetilde{\Phi_1})\backslash (N(w_1)\cap \widetilde{\Phi_1})\subset -G$, i.e. $w_1\leq_{G\cap \widetilde{\Phi_1}} w_2$.
\end{proof}

\begin{corollary}\label{block}
Let $P$ be a block of a tope poset of $(\widetilde{\Phi},-,cx)$. Then $P$ is isomorphic to the twisted weak order on some Coxeter group $W'$.
\end{corollary}

\begin{proof}
By Corollary \ref{twotypes}, there are two possible types of the base tope (hemisapce).

Case I: The base hemispace is of the form $N(w)\uplus -(\widehat{\Phi}\backslash N(w))$ where $w\in \widetilde{W}$ or $w\in \widetilde{W}_l$. Then by Theorem  \ref{infiniteword}, $N(w)=u\cdot P(\Psi^+,\Delta_1,\emptyset)^{\wedge}$ where $u\in \widetilde{W}$. (Note if $w\in \widetilde{W}$, then $\Delta_1=\Delta$ and $P(\Psi^+,\Delta_1,\emptyset)^{\wedge}=\emptyset.$) Suppose that $u=e$, the assertion follows from Lemma \ref{torder}.

Now assume that $u\neq e$. Note that the map $H\mapsto uH$ is a bijection from the set of hemispaces to itself. One easily sees  that the tope poset based at $H$ is isomorphic to the tope poset based at $wH$ by noting $G_1\Delta H\subset G_2\Delta H\Leftrightarrow uG_1\Delta uH\subset uG_2\Delta uH$. Therefore the assertion also holds in this situation.

Case II. The base hemispace is of the form $-N(w)\uplus (\widehat{\Phi}\backslash N(w))$ where $w\in \widetilde{W}$ or $w\in \widetilde{W}_l$. One easily sees that such a tope poset is the opposite poset of the tope poset based at $N(w)\uplus -(\widehat{\Phi}\backslash N(w))$. Since the opposite poset of the twisted weak order $(W,\leq_B')$ is isomorphic to $(W,\leq_{\Phi^+\backslash B'})$, the assertion also holds in this case.
\end{proof}

\noindent
\emph{Proof of Theorem \ref{intervallattice}}. By Corollary \ref{block}, $[H_1, H_2]$ is isomorphic to a closed interval of the twisted weak order of some Coxeter group $W'$. By Corollary 4.2 of \cite{orderpaper}, it is also isomorphic to a closed interval of the ordinary weak order of $W'$. Since the ordinary weak order $(W',\leq_{\emptyset}')$ is a complete meet semilattice, such an interval is a lattice.

\begin{remark}
Note that the example in Section 4 of \cite{wang} shows that in general  the tope poset of the oriented matroid coming from an affine root system is not a lattice.
\end{remark}

To end this section, we sketch (part of) the Hasse diagram of a specific tope poset of $\widetilde{\Phi}$ of type $\widetilde{A}_2$ with the base hemispace $\widetilde{\Phi}^+$. This particular tope poset is isomorphic to the poset of 2 closure biclosed sets in $\widetilde{\Phi}^+$ (but in general there is no such isomorphism by our results above) and thus is a lattice by \cite{wang}. The two simple roots in $\Phi^+$ are denoted by $\alpha,\beta.$

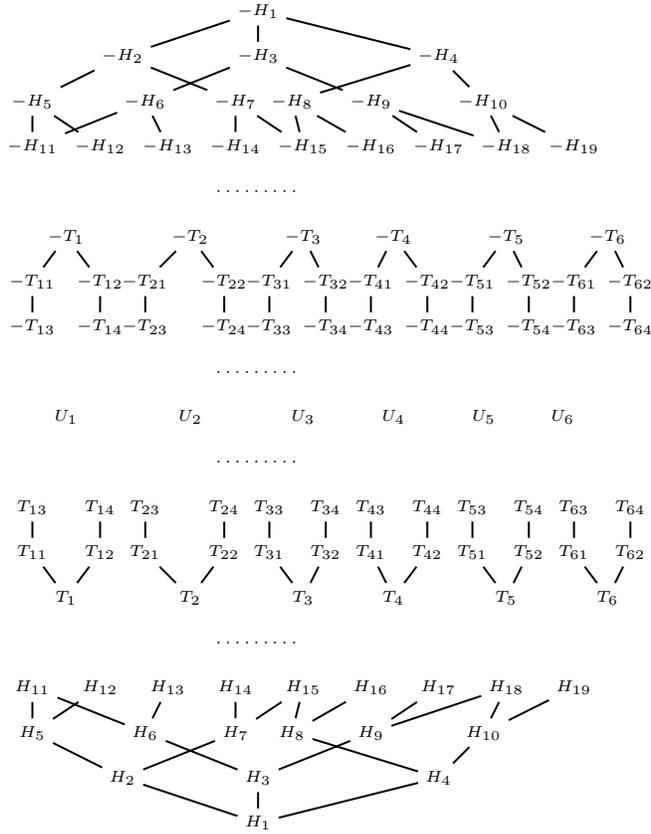
\begin{tikzpicture}[scale=.15]
\tiny
  \node (a) at (0,0) {$H_1$};
   \node (b) at (-12,4) {$H_2$};
  \node (c) at (0,4) {$H_3$};
 \node (d) at (16,4) {$H_4$};
 \node (e) at (-20,8) {$H_5$};
 \node (f) at (-10,8) {$H_6$};
 \node (g) at (-2,8) {$H_7$};
 \node (h) at  (3,8) {$H_8$};
 \node (i) at (10,8)  {$H_9$};
 \node (j) at (20,8)  {$H_{10}$};
 \node (k)  at (-20,12) {$H_{11}$};
 \node (l) at (-14,12)   {$H_{12}$};
 \node (m) at (-8,12)   {$H_{13}$};
 \node  (n) at (-2,12)  {$H_{14}$};
 \node  (o)  at (4,12)  {$H_{15}$};
 \node  (p)  at (10,12)  {$H_{16}$};
 \node  (q)  at (16,12)  {$H_{17}$};
 \node  (r)  at (22,12)  {$H_{18}$};
 \node  (s)  at (28,12)  {$H_{19}$};
 \node  (t)  at (0,16)  {$\cdots\cdots\cdots$};
 \node  (u)  at  (-17,20)   {$T_1$};
 \node  (v)  at  (-6,20)   {$T_2$};
 \node  (w)  at  (4,20)   {$T_3$};
 \node  (x)  at  (12,20)   {$T_4$};
  \node  (y)  at  (22,20)   {$T_5$};
  \node  (z)  at  (31,20)   {$T_6$};
  \node  (aa)  at  (-20,24)   {$T_{11}$};
   \node  (ab)  at  (-14,24)   {$T_{12}$};
   \node  (ac)  at  (-20,28)   {$T_{13}$};
   \node  (ad)  at  (-14,28)   {$T_{14}$};
    \node  (ae)  at  (-10,24)   {$T_{21}$};
   \node  (af)  at  (-3,24)   {$T_{22}$};
   \node  (ag)  at  (-10,28)   {$T_{23}$};
   \node  (ah)  at  (-3,28)   {$T_{24}$};
   \node  (ai)  at  (1,24)   {$T_{31}$};
   \node  (aj)  at  (6,24)   {$T_{32}$};
   \node  (ak)  at  (1,28)   {$T_{33}$};
   \node  (al)  at  (6,28)   {$T_{34}$};
   \node  (am)  at  (10,24)   {$T_{41}$};
   \node  (an)  at  (15,24)   {$T_{42}$};
   \node  (ao)  at  (10,28)   {$T_{43}$};
   \node  (ap)  at  (15,28)   {$T_{44}$};
   \node  (aq)  at  (19,24)   {$T_{51}$};
   \node  (ar)  at  (24,24)   {$T_{52}$};
   \node  (as)  at  (19,28)   {$T_{53}$};
   \node  (at)  at  (24,28)   {$T_{54}$};
    \node  (au)  at  (28,24)   {$T_{61}$};
   \node  (av)  at  (33,24)   {$T_{62}$};
   \node  (aw)  at  (28,28)   {$T_{63}$};
   \node  (ax)  at  (33,28)   {$T_{64}$};
   \node  (ay)  at (0,32)  {$\cdots\cdots\cdots$};
   \node  (az)  at (-17,36)  {$U_1$};
   \node   (ba) at (-6,36)      {$U_2$};
   \node   (bb) at (4,36)      {$U_3$};
   \node   (bc) at (12,36)      {$U_4$};
   \node   (bd) at (20,36)      {$U_5$};
   \node   (be) at (27,36)      {$U_6$};
   \node   (bf) at (0,40)      {$\cdots\cdots\cdots$};
   \node  (bg)  at  (-20,48)   {$-T_{11}$};
   \node  (bh)  at  (-14,48)   {$-T_{12}$};
   \node  (bi)  at  (-20,44)   {$-T_{13}$};
   \node  (bj)  at  (-14,44)   {$-T_{14}$};
   \node  (bk)  at  (-10,48)   {$-T_{21}$};
   \node  (bl)  at  (-3,48)   {$-T_{22}$};
   \node  (bm)  at  (-10,44)   {$-T_{23}$};
   \node  (bn)  at  (-3,44)   {$-T_{24}$};
   \node  (bo)  at  (1,48)   {$-T_{31}$};
   \node  (bp)  at  (6,48)   {$-T_{32}$};
   \node  (bq)  at  (1,44)   {$-T_{33}$};
   \node  (br)  at  (6,44)   {$-T_{34}$};
   \node  (bs)  at  (10,48)   {$-T_{41}$};
   \node  (bt)  at  (15,48)   {$-T_{42}$};
   \node  (bu)  at  (10,44)   {$-T_{43}$};
   \node  (bv)  at  (15,44)   {$-T_{44}$};
   \node  (bw)  at  (19,48)   {$-T_{51}$};
   \node  (bx)  at  (24,48)   {$-T_{52}$};
   \node  (by)  at  (19,44)   {$-T_{53}$};
   \node  (bz)  at  (24,44)   {$-T_{54}$};
    \node  (ca)  at  (28,48)   {$-T_{61}$};
   \node  (cb)  at  (33,48)   {$-T_{62}$};
   \node  (cc)  at  (28,44)   {$-T_{63}$};
   \node  (cd)  at  (33,44)   {$-T_{64}$};
   \node  (uu)  at  (-17,52)   {$-T_1$};
 \node  (vv)  at  (-6,52)   {$-T_2$};
 \node  (ww)  at  (4,52)   {$-T_3$};
 \node  (xx)  at  (12,52)   {$-T_4$};
  \node  (yy)  at  (22,52)   {$-T_5$};
  \node  (zz)  at  (31,52)   {$-T_6$};
  \node (dots) at (0,56) {$\cdots\cdots\cdots$};
    \node (aaa) at (0,72) {$-H_1$};
   \node (bbb) at (-12,68) {$-H_2$};
  \node (ccc) at (0,68) {$-H_3$};
 \node (ddd) at (16,68) {$-H_4$};
 \node (eee) at (-20,64) {$-H_5$};
 \node (fff) at (-10,64) {$-H_6$};
 \node (ggg) at (-2,64) {$-H_7$};
 \node (hhh) at  (3,64) {$-H_8$};
 \node (iii) at (10,64)  {$-H_9$};
 \node (jjj) at (20,64)  {$-H_{10}$};
 \node (kkk)  at (-20,60) {$-H_{11}$};
 \node (lll) at (-14,60)   {$-H_{12}$};
 \node (mmm) at (-8,60)   {$-H_{13}$};
 \node  (nnn) at (-2,60)  {$-H_{14}$};
 \node  (ooo)  at (4,60)  {$-H_{15}$};
 \node  (ppp)  at (10,60)  {$-H_{16}$};
 \node  (qqq)  at (16,60)  {$-H_{17}$};
 \node  (rrr)  at (22,60)  {$-H_{18}$};
 \node  (sss)  at (28,60)  {$-H_{19}$};
   \draw (ccc) -- (aaa);
  \draw (bbb) -- (aaa);
  \draw (ddd) -- (aaa);
  \draw (eee)--(bbb);
  \draw (fff)--(ccc);
  \draw (ggg)--(bbb);
  \draw (hhh)--(ddd);
  \draw (iii)--(ccc);
  \draw  (jjj)--(ddd);
  \draw (kkk)--(eee);
  \draw  (kkk)--(fff);
  \draw  (lll)--(eee);
  \draw (mmm)--(fff);
  \draw  (nnn)--(ggg);
  \draw  (ooo)--(ggg);
  \draw  (ooo)--(hhh);
  \draw  (ppp)--(hhh);
  \draw  (qqq)--(iii);
  \draw   (rrr)--(iii);
  \draw   (rrr)--(jjj);
  \draw   (sss)--(jjj);

  \draw  (ca)--(zz);
  \draw   (cb)--(zz);
  \draw  (bw)--(yy);
  \draw   (bx)--(yy);
  \draw  (bs)--(xx);
  \draw  (bt)--(xx);
  \draw  (bo)--(ww);
  \draw  (bp)--(ww);
  \draw  (bk)--(vv);
  \draw  (bl)--(vv);
  \draw (bg)--(uu);
  \draw  (bh)--(uu);
   \draw  (by)--(bw);
   \draw  (bz)--(bx);
   \draw  (cc)--(ca);
   \draw  (cd)--(cb);
   \draw (bu)--(bs);
   \draw (bv)--(bt);
   \draw  (bq)--(bo);
   \draw  (br)--(bp);
   \draw (bm)--(bk);
   \draw  (bn)--(bl);
   \draw (bi)--(bg);
   \draw (bj)--(bh);
  \draw (c) -- (a);
  \draw (b) -- (a);
  \draw (d) -- (a);
  \draw (e)--(b);
  \draw (f)--(c);
  \draw (g)--(b);
  \draw (h)--(d);
  \draw (i)--(c);
  \draw  (j)--(d);
  \draw (k)--(e);
  \draw  (k)--(f);
  \draw  (l)--(e);
  \draw (m)--(f);
  \draw  (n)--(g);
  \draw  (o)--(g);
  \draw  (o)--(h);
  \draw  (p)--(h);
  \draw  (q)--(i);
  \draw   (r)--(i);
  \draw   (r)--(j);
  \draw   (s)--(j);
  \draw  (aa)--(u);
  \draw   (ab)--(u);
  \draw  (ac)--(aa);
  \draw  (ad)--(ab);
  \draw  (ae)--(v);
  \draw  (af)--(v);
  \draw  (ag)--(ae);
  \draw  (ah)--(af);
  \draw  (ai)--(w);
  \draw  (aj)--(w);
  \draw  (ak)--(ai);
  \draw   (al)--(aj);
  \draw  (am)--(x);
  \draw   (an)--(x);
  \draw  (ao)--(am);
  \draw  (ap)--(an);
  \draw   (aq)--(y);
  \draw   (ar)--(y);
  \draw   (as)--(aq);
  \draw   (at)--(ar);
  \draw   (au)--(z);
  \draw   (av)--(z);
  \draw    (aw)--(au);
  \draw    (ax)--(av);

\end{tikzpicture}

\tiny
$H_1=-\widehat{\Phi};$

$H_2=\{\alpha\}\uplus -(\widehat{\Phi}\backslash \{\alpha\})$;

$H_3=\{\beta\}\uplus -(\widehat{\Phi}\backslash \{\beta\});$

$H_4=\{\delta-\alpha-\beta\}\uplus -(\widehat{\Phi}\backslash \{\delta-\alpha-\beta\})$;

$H_5=\{\alpha, \alpha+\beta\}\uplus -(\widehat{\Phi}\backslash \{\alpha,\alpha+\beta\})$;

$H_6=\{\beta, \alpha+\beta\}\uplus -(\widehat{\Phi}\backslash \{\beta,\alpha+\beta\});$

$H_7=\{\alpha, \delta-\beta\}\uplus -(\widehat{\Phi}\backslash \{\alpha,\delta-\beta\})$;

$H_8=\{\delta-\alpha-\beta,\delta-\beta\}\uplus -(\widehat{\Phi}\backslash \{\delta-\alpha-\beta,\delta-\beta\})$;

$H_9=\{\beta, \delta-\alpha\}\uplus -(\widehat{\Phi}\backslash \{\beta,\delta-\alpha\})$;

$H_{10}=\{\delta-\alpha-\beta,\delta-\alpha\}\uplus -(\widehat{\Phi}\backslash \{\delta-\alpha-\beta,\delta-\alpha\})$;

$H_{11}=\{\alpha, \alpha+\beta, \beta\}\uplus -(\widehat{\Phi}\backslash \{\alpha,\alpha+\beta, \beta\})$;

$H_{12}=\{\alpha, \alpha+\beta, \alpha+\delta\}\uplus -(\widehat{\Phi}\backslash \{\alpha,\alpha+\beta, \alpha+\delta\})$;

$H_{13}=\{\beta, \alpha+\beta, \beta+\delta\}\uplus -(\widehat{\Phi}\backslash \{\beta,\alpha+\beta, \beta+\delta\})$;

$H_{14}=\{\alpha, \delta-\beta, \alpha+\delta\}\uplus -(\widehat{\Phi}\backslash \{\alpha,\delta-\beta, \alpha+\delta\})$;

$H_{15}=\{\alpha, \delta-\beta, \delta-\alpha-\delta\}\uplus -(\widehat{\Phi}\backslash \{\alpha,\delta-\beta,\delta-\alpha-\delta\})$;

$H_{16}=\{2\delta-\alpha-\delta, \delta-\beta, \delta-\alpha-\delta\}\uplus -(\widehat{\Phi}\backslash \{2\delta-\alpha-\delta,\delta-\beta,\delta-\alpha-\delta\})$;

$H_{17}=\{\beta,\delta-\alpha, \beta+\delta\}\uplus -(\widehat{\Phi}\backslash \{\beta,\delta-\alpha, \beta+\delta\})$;

$H_{18}=\{\beta,\delta-\alpha, \delta-\beta-\delta\}\uplus -(\widehat{\Phi}\backslash \{\beta,\delta-\alpha, \delta-\beta-\delta\})$;

$H_{19}=\{2\delta-\beta-\delta,\delta-\alpha, \delta-\beta-\delta\}\uplus -(\widehat{\Phi}\backslash \{2\delta-\beta-\delta,\delta-\alpha, \delta-\beta-\delta\})$;

$T_1=\widehat{\alpha}\uplus\widehat{\alpha+\beta}\uplus -(\widehat{\Phi}\backslash (\widehat{\alpha}\uplus\widehat{\alpha+\beta}))$;

$T_2=\widehat{\beta}\uplus\widehat{\alpha+\beta}\uplus -(\widehat{\Phi}\backslash (\widehat{\beta}\uplus\widehat{\alpha+\beta}))$;

$T_3=\widehat{\beta}\uplus\widehat{-\alpha}\uplus -(\widehat{\Phi}\backslash (\widehat{\beta}\uplus\widehat{-\alpha}))$;

$T_4=\widehat{-\beta-\alpha}\uplus\widehat{-\alpha}\uplus -(\widehat{\Phi}\backslash (\widehat{-\beta-\alpha}\uplus\widehat{-\alpha}))$;

$T_5=\widehat{-\beta-\alpha}\uplus\widehat{-\beta}\uplus -(\widehat{\Phi}\backslash (\widehat{-\beta-\alpha}\uplus\widehat{-\beta}))$;

$T_6=\widehat{\alpha}\uplus\widehat{-\beta}\uplus -(\widehat{\Phi}\backslash (\widehat{\alpha}\uplus\widehat{-\beta}))$;

$T_{11}=\widehat{\alpha}\uplus\widehat{\alpha+\beta}\uplus \{\beta\}\uplus -(\widehat{\Phi}\backslash (\widehat{\alpha}\uplus\widehat{\alpha+\beta}\uplus \{\beta\}))$;

$T_{12}=\widehat{\alpha}\uplus\widehat{\alpha+\beta}\uplus \{\delta-\beta\}\uplus -(\widehat{\Phi}\backslash (\widehat{\alpha}\uplus\widehat{\alpha+\beta}\uplus \{\delta-\beta\}))$

$T_{13}=\widehat{\alpha}\uplus\widehat{\alpha+\beta}\uplus \{\beta, \beta+\delta\}\uplus -(\widehat{\Phi}\backslash (\widehat{\alpha}\uplus\widehat{\alpha+\beta}\uplus \{\beta, \beta+\delta\}))$;

$T_{14}=\widehat{\alpha}\uplus\widehat{\alpha+\beta}\uplus \{\delta-\beta, 2\delta-\beta\}\uplus -(\widehat{\Phi}\backslash (\widehat{\alpha}\uplus\widehat{\alpha+\beta}\uplus \{\delta-\beta, 2\delta-\beta\}))$;

$T_{21}=\widehat{\beta}\uplus\widehat{\alpha+\beta}\uplus \{\alpha\}\uplus -(\widehat{\Phi}\backslash (\widehat{\beta}\uplus\widehat{\alpha+\beta}\uplus \{\alpha\}))$;

$T_{22}=\widehat{\beta}\uplus\widehat{\alpha+\beta}\uplus \{\delta-\alpha\}\uplus -(\widehat{\Phi}\backslash (\widehat{\beta}\uplus\widehat{\alpha+\beta}\uplus \{\delta-\alpha\}))$;

$T_{23}=\widehat{\beta}\uplus\widehat{\alpha+\beta}\uplus \{\alpha, \alpha+\delta\}\uplus -(\widehat{\Phi}\backslash (\widehat{\beta}\uplus\widehat{\alpha+\beta}\uplus \{\alpha, \alpha+\delta\}))$;

$T_{24}=\widehat{\beta}\uplus\widehat{\alpha+\beta}\uplus \{\delta-\alpha, 2\delta-\alpha\}\uplus -(\widehat{\Phi}\backslash (\widehat{\beta}\uplus\widehat{\alpha+\beta}\uplus \{\delta-\alpha, 2\delta-\alpha\}))$;

$T_{31}=\widehat{\beta}\uplus\widehat{-\alpha}\uplus \{\alpha+\beta\}\uplus -(\widehat{\Phi}\backslash (\widehat{\beta}\uplus\widehat{-\alpha}\uplus \{\alpha+\beta\}))$;

$T_{32}=\widehat{\beta}\uplus\widehat{-\alpha}\uplus \{\delta-\alpha-\beta\}\uplus -(\widehat{\Phi}\backslash (\widehat{\beta}\uplus\widehat{-\alpha}\uplus \{\delta-\alpha-\beta\}))$;

$T_{33}=\widehat{\beta}\uplus\widehat{-\alpha}\uplus \{\alpha+\beta, \alpha+\beta+\delta\}\uplus -(\widehat{\Phi}\backslash (\widehat{\beta}\uplus\widehat{-\alpha}\uplus \{\alpha+\beta, \alpha+\beta+\delta\}))$;

$T_{34}=\widehat{\beta}\uplus\widehat{-\alpha}\uplus \{\delta-\alpha-\beta, 2\delta-\alpha-\beta\}\uplus -(\widehat{\Phi}\backslash (\widehat{\beta}\uplus\widehat{-\alpha}\uplus \{\delta-\alpha-\beta, 2\delta-\alpha-\beta\}))$;

$T_{41}=\widehat{-\alpha-\beta}\uplus\widehat{-\alpha}\uplus \{\beta\}\uplus -(\widehat{\Phi}\backslash (\widehat{-\alpha-\beta}\uplus\widehat{-\alpha}\uplus \{\beta\}))$;

$T_{42}=\widehat{-\alpha-\beta}\uplus\widehat{-\alpha}\uplus \{\delta-\beta\}\uplus -(\widehat{\Phi}\backslash (\widehat{\-\alpha-\beta}\uplus\widehat{-\alpha}\uplus \{\delta-\beta\}))$;

$T_{43}=\widehat{-\alpha-\beta}\uplus\widehat{-\alpha}\uplus \{\beta, \beta+\delta\}\uplus -(\widehat{\Phi}\backslash (\widehat{-\alpha-\beta}\uplus\widehat{-\alpha}\uplus \{\beta, \beta+\delta\}))$;

$T_{44}=\widehat{-\alpha-\beta}\uplus\widehat{-\alpha}\uplus \{\delta-\beta, 2\delta-\beta\}\uplus -(\widehat{\Phi}\backslash (\widehat{-\alpha-\beta}\uplus\widehat{-\alpha}\uplus \{\delta-\beta, 2\delta-\beta\}))$;

$T_{51}=\widehat{-\alpha-\beta}\uplus\widehat{-\beta}\uplus \{\alpha\}\uplus -(\widehat{\Phi}\backslash (\widehat{-\alpha-\beta}\uplus\widehat{-\beta}\uplus \{\alpha\}))$;

$T_{52}=\widehat{-\alpha-\beta}\uplus\widehat{-\beta}\uplus \{\delta-\alpha\}\uplus -(\widehat{\Phi}\backslash (\widehat{\-\alpha-\beta}\uplus\widehat{-\beta}\uplus \{\delta-\alpha\}))$;

$T_{53}=\widehat{-\alpha-\beta}\uplus\widehat{-\beta}\uplus \{\alpha, \alpha+\delta\}\uplus -(\widehat{\Phi}\backslash (\widehat{-\alpha-\beta}\uplus\widehat{-\beta}\uplus \{\alpha, \alpha+\delta\}))$;

$T_{54}=\widehat{-\alpha-\beta}\uplus\widehat{-\beta}\uplus \{\delta-\alpha, 2\delta-\alpha\}\uplus -(\widehat{\Phi}\backslash (\widehat{-\alpha-\beta}\uplus\widehat{-\beta}\uplus \{\delta-\alpha, 2\delta-\alpha\}))$;

$T_{61}=\widehat{\alpha}\uplus\widehat{-\beta}\uplus \{\alpha+\beta\}\uplus -(\widehat{\Phi}\backslash (\widehat{\alpha}\uplus\widehat{-\beta}\uplus \{\alpha+\beta\}))$;

$T_{62}=\widehat{\alpha}\uplus\widehat{-\beta}\uplus \{\delta-\alpha-\beta\}\uplus -(\widehat{\Phi}\backslash (\widehat{\alpha}\uplus\widehat{-\beta}\uplus \{\delta-\alpha-\beta\}))$;

$T_{63}=\widehat{\alpha}\uplus\widehat{-\beta}\uplus \{\alpha+\beta, \alpha+\beta+\delta\}\uplus -(\widehat{\Phi}\backslash (\widehat{\alpha}\uplus\widehat{-\beta}\uplus \{\alpha+\beta, \alpha+\beta+\delta\}))$;

$T_{64}=\widehat{\alpha}\uplus\widehat{-\beta}\uplus \{\delta-\alpha-\beta, 2\delta-\alpha-\beta\}\uplus -(\widehat{\Phi}\backslash (\widehat{\alpha}\uplus\widehat{-\beta}\uplus \{\delta-\alpha-\beta, 2\delta-\alpha-\beta\}))$;

$U_1=\widehat{\{\alpha,\beta,\alpha+\beta\}}\uplus -(\widehat{\Phi}\backslash \widehat{\{\alpha,\beta,\alpha+\beta\}})$;

$U_2=\widehat{\{-\alpha,\beta,\alpha+\beta\}}\uplus -(\widehat{\Phi}\backslash \widehat{\{-\alpha,\beta,\alpha+\beta\}})$;

$U_3=\widehat{\{-\alpha,\beta,-\alpha-\beta\}}\uplus -(\widehat{\Phi}\backslash \widehat{\{-\alpha,\beta,-\alpha-\beta\}})$;

$U_4=\widehat{\{-\alpha,-\beta,-\alpha-\beta\}}\uplus -(\widehat{\Phi}\backslash \widehat{\{-\alpha,-\beta,-\alpha-\beta\}})$;

$U_5=\widehat{\{\alpha,-\beta,-\alpha-\beta\}}\uplus -(\widehat{\Phi}\backslash \widehat{\{\alpha,-\beta,-\alpha-\beta\}})$;

$U_6=\widehat{\{\alpha,-\beta,\alpha+\beta\}}\uplus -(\widehat{\Phi}\backslash \widehat{\{\alpha,-\beta,\alpha+\beta\}})$;

\normalsize

\section{Acknowledgement}

The author acknowledges the support from Guangdong  Natural Science Foundation  Project 2018A030313581. A small part of Section \ref{omsec} is based on a chapter of the author's dissertation. The author thanks Matthew Dyer for suggesting the problems in Section \ref{localfinite} and \ref{fixlength} and useful discussions.

\end{document}